\documentclass[oneside,english,tac.cls]{amsart}
\usepackage[LGR,T1]{fontenc}
\usepackage[latin9]{inputenc}
\usepackage{mathrsfs}
\usepackage{mathtools}
\usepackage{amstext}
\usepackage{amsthm}
\usepackage{amssymb}
\usepackage{stmaryrd}
\usepackage[all]{xy}

\makeatletter

\newcommand{\noun}[1]{\textsc{#1}}
\DeclareRobustCommand{\greektext}{%
  \fontencoding{LGR}\selectfont\def\encodingdefault{LGR}}
\DeclareRobustCommand{\textgreek}[1]{\leavevmode{\greektext #1}}
\ProvideTextCommand{\~}{LGR}[1]{\char126#1}

\numberwithin{equation}{section}
\numberwithin{figure}{section}
\theoremstyle{plain}
\newtheorem{thm}{\protect\theoremname}
\theoremstyle{definition}
\newtheorem{example}[thm]{\protect\examplename}
\theoremstyle{definition}
\newtheorem{defn}[thm]{\protect\definitionname}
\theoremstyle{remark}
\newtheorem{rem}[thm]{\protect\remarkname}
\theoremstyle{plain}
\newtheorem{prop}[thm]{\protect\propositionname}
\theoremstyle{plain}
\newtheorem{lem}[thm]{\protect\lemmaname}
\theoremstyle{plain}
\newtheorem{cor}[thm]{\protect\corollaryname}

\usepackage[dvipsnames,svgnames,x11names,hyperref]{xcolor}
\usepackage[colorlinks=true,linkcolor=Dark Red, citecolor=Dark Green,linktoc=all]{hyperref}
\usepackage{lmodern}
\renewcommand*{\epsilon}{\varepsilon}
\usepackage{amssymb}
\usepackage{amsfonts}
\usepackage{egothic}

\usepackage[T1]{fontenc}

\usepackage{url}
\usepackage{amsthm}
\usepackage{aliascnt}
\usepackage{lipsum}
\usepackage{mathrsfs}
\usepackage{esint}
\usepackage{url}
\usepackage{amsmath}
\usepackage{dsfont}
\usepackage{bbm}
\usepackage{pdflscape}
\usepackage{bbm}
\usepackage{bussproofs}

\usepackage{tikz}
\tikzset{node distance=2.5cm, auto}
\usetikzlibrary{positioning}

\newcommand{\myar}[2]{\ar^-{#1}[#2]}
\newcommand{\myard}[2]{\ar_-{#1}[#2]}

\makeatletter
\def\matrixobject@{%
  \edef \next@{={\DirectionfromtheDirection@ }}%
  \expandafter \toks@ \next@ \plainxy@
  \let\xy@@ix@=\xyq@@toksix@
  \xyFN@ \OBJECT@}
\let\xy@entry@@norm=\entry@@norm
\def\entry@@norm@patched{%
  \let\object@=\matrixobject@
  \xy@entry@@norm }
\AtBeginDocument{\let\entry@@norm\entry@@norm@patched}
\makeatother

\newcommand{\twocong}[2][0.5]{\ar@{}[#2] \save ?(#1)*{\cong}\restore}
\newcommand{\twoeq}[2][0.5]{\ar@{}[#2] \save ?(#1)*{=}\restore}
\newcommand{\ltwocell}[3][0.5]{\ar@{}[#2] \ar@{=>}?(#1)+/r 0.15cm/;?(#1)+/l 0.15cm/^{#3}}
\newcommand{\rtwocell}[3][0.5]{\ar@{}[#2] \ar@{=>}?(#1)+/l 0.15cm/;?(#1)+/r 0.15cm/^{#3}}
\newcommand{\utwocell}[3][0.5]{\ar@{}[#2] \ar@{=>}?(#1)+/d  0.15cm/;?(#1)+/u 0.15cm/_{#3}}
\newcommand{\dtwocell}[3][0.5]{\ar@{}[#2] \ar@{=>}?(#1)+/u  0.15cm/;?(#1)+/d 0.15cm/^{#3}}
\newcommand{\ultwocell}[3][0.5]{\ar@{}[#2] \ar@{=>}?(#1)+/dr  0.15cm/;?(#1)+/ul 0.15cm/^{#3}}
\newcommand{\urtwocell}[3][0.5]{\ar@{}[#2] \ar@{=>}?(#1)+/dl  0.15cm/;?(#1)+/ur 0.15cm/^{#3}}
\newcommand{\dltwocell}[3][0.5]{\ar@{}[#2] \ar@{=>}?(#1)+/ur  0.15cm/;?(#1)+/dl 0.15cm/^{#3}}
\newcommand{\drtwocell}[3][0.5]{\ar@{}[#2] \ar@{=>}?(#1)+/ul  0.15cm/;?(#1)+/dr 0.15cm/^{#3}}

\makeatletter
\renewcommand{\tocsection}[3]{%
  \indentlabel{\@ifnotempty{#2}{\bfseries\ignorespaces#1 #2\quad}}\bfseries#3}
\renewcommand{\tocsubsection}[3]{%
  \indentlabel{\@ifnotempty{#2}{\ignorespaces#1 #2\quad}}#3}

\newcommand\@dotsep{4.5}
\def\@tocline#1#2#3#4#5#6#7{\relax
  \ifnum #1>\c@tocdepth 
  \else
    \par \addpenalty\@secpenalty\addvspace{#2}%
    \begingroup \hyphenpenalty\@M
    \@ifempty{#4}{%
      \@tempdima\csname r@tocindent\number#1\endcsname\relax
    }{%
      \@tempdima#4\relax
    }%
    \parindent\z@ \leftskip#3\relax \advance\leftskip\@tempdima\relax
    \rightskip\@pnumwidth plus1em \parfillskip-\@pnumwidth
    #5\leavevmode\hskip-\@tempdima{#6}\nobreak
    \leaders\hbox{$\m@th\mkern \@dotsep mu\hbox{.}\mkern \@dotsep mu$}\hfill
    \nobreak
    \hbox to\@pnumwidth{\@tocpagenum{\ifnum#1=1\bfseries\fi#7}}\par
    \nobreak
    \endgroup
  \fi}
\AtBeginDocument{%
\expandafter\renewcommand\csname r@tocindent0\endcsname{0pt}
}
\def\l@subsection{\@tocline{2}{0pt}{2.5pc}{5pc}{}}
\def\l@section{\@tocline{2}{3pt}{2.5pc}{5pc}{}}
\makeatother

\usepackage{lipsum}

\numberwithin{thm}{section}

\makeatother

\usepackage{babel}
\providecommand{\corollaryname}{Corollary}
\providecommand{\definitionname}{Definition}
\providecommand{\examplename}{Example}
\providecommand{\lemmaname}{Lemma}
\providecommand{\propositionname}{Proposition}
\providecommand{\remarkname}{Remark}
\providecommand{\theoremname}{Theorem}

\begin{document}
\title{Lax Familial Representability and Lax Generic Factorizations}
\begin{abstract}
A classical result due to Diers shows that a copresheaf $F\colon\mathcal{A}\to\mathbf{Set}$
on a category $\mathcal{A}$ is a coproduct of representables precisely
when each connected component of $F$'s category of elements has an
initial object. Most often, this condition is imposed on a copresheaf
of the form $\mathcal{B}\left(X,T-\right)$ for a functor $T\colon\mathcal{A}\to\mathcal{B}$,
in which case this property says that $T$ admits generic factorizations
at $X$, or equivalently that $T$ is familial at $X$.

Here we generalize these results to the two-dimensional setting, replacing
$\mathcal{A}$ with an arbitrary bicategory $\mathscr{A}$, and $\mathbf{Set}$
with $\mathbf{Cat}$. In this two-dimensional setting, simply asking
that a pseudofunctor $F\colon\mathscr{A}\to\mathbf{Cat}$ be a coproduct
of representables is often too strong of a condition. Instead, we
will only ask that $F$ be a lax conical colimit of representables.
This in turn allows for the weaker notion of lax generic factorizations
(and lax familial representability) for pseudofunctors of bicategories
$T\colon\mathscr{A}\to\mathscr{B}$.

We also compare our lax familial pseudofunctors to Weber's familial
2-functors, finding our description is more general (not requiring
a terminal object in $\mathscr{A}$), though essentially equivalent
when a terminal object does exist. Moreover, our description of lax
generics allows for an equivalence between lax generic factorizations
and lax familial representability.

Finally, we characterize our lax familial pseudofunctors as right
lax $\mathsf{F}$-adjoints followed by locally discrete fibrations
of bicategories, which in turn yields a simple definition of parametric
right adjoint pseudofunctors.
\end{abstract}

\author{Charles Walker}
\thanks{The author gratefully acknowledges the support of an Australian Government
Research Training Program Scholarship. }
\address{Department of Mathematics and Statistics, Masaryk University, Kotl{\'a}{\v r}sk{\'a}
2, Brno 61137, Czech Republic }
\email{\tt{walker@math.muni.cz}}

\maketitle
\tableofcontents{}

\section{Introduction}

This paper is concerned with the notion of \emph{familial representability},
a condition first studied in detail by Diers \cite{Diers} for 1-categories,
and how the theory of familial representability can be generalized
to the two-dimensional setting.

\subsection{Familial representability}

Given a category $\mathcal{A}$ and presheaf $F\colon\mathcal{A}\to\mathbf{Set}$
(actually a ``copresheaf'', we suppress the ``co'' for brevity),
it is often useful to know whether this presheaf is a coproduct of
representable presheaves; meaning 
\[
F\cong\sum_{m\in\mathfrak{M}}\mathcal{A}\left(P_{m},-\right)
\]
for some set $\mathfrak{M}$ and function $P_{\left(-\right)}\colon\mathfrak{M}\to\mathcal{A}$.
Such presheaves have a straightforward characterization: a presheaf
$F$ is a coproduct of representables precisely when each connected
component of its category of elements, denoted $\textnormal{el }F$,
has an initial object. Expressing this condition in more detail, this
means that for any object $\left(D,w\right)$ in $\textnormal{el }F$
there exists an object $\left(A,x\right)$ and morphism $k\colon\left(A,x\right)\to\left(D,w\right)$
where $\left(A,x\right)$ satisfies the following property (which
defines initial objects in a connected component): for any diagram
in $\textnormal{el }F$ as below 
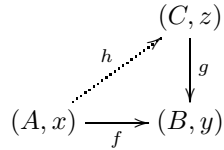
\[
\xymatrix{\; & \left(C,z\right)\myar{g}{d}\\
\left(A,x\right)\myard{f}{r}\ar@{..>}[ur]^{h} & \left(B,y\right)
}
\]
there exists a unique morphism $h\colon\left(A,x\right)\to\left(C,z\right)$,
and consequently the above triangle commutes.

Of particular interest is the case where $F$ is of the form $\mathcal{B}\left(X,T-\right)$
for a functor $T\colon\mathcal{A}\to\mathcal{B}$ between categories
$\mathcal{A}$ and $\mathcal{B}$. This condition, first studied in
detail by Diers \cite{Diers}, asks that we have an isomorphism $\mathcal{B}\left(X,T-\right)\cong\sum_{m\in\mathfrak{M}}\mathcal{A}\left(P_{m},-\right)$
and generalizes $T$ having a left adjoint. Thus such a $T$ is often
referred to as a functor having a left multiadjoint \cite{Diers},
however we will simply refer to such a $T$ as \emph{familial}. It
is also worth noting that the functors $T$ with this property may
be seen as the admissible maps against the KZ pseudomonad \cite{WalkerYonedaKZ}
freely adding sums.\footnote{Under this characterization one would suitably replace categories
with their opposites, as the condition given concerns \emph{co}presheaves.}

If we specialize the above to this case, we see that asking $\mathcal{B}\left(X,T-\right)$
be familial amounts to asking that for any $w\colon X\to TD$ there
exists an $x\colon X\to TA$ and $k\colon A\to D$ such that $w=Tk\cdot x$,
and $x$ is ``generic'' meaning that it satisfies the following
property: given any commuting square as on the left below 
\[
\xymatrix{X\myar{z}{r}\ar[d]_{x} & TB\ar[d]^{Tg} &  &  & X\myar{z}{r}\ar[d]_{x} & TB\ar[d]^{Tg}\\
TA\myard{Tf}{r} & TC & \; &  & TA\myard{Tf}{r}\ar[ur]|-{Th} & TC
}
\]
there exists a unique $h\colon A\to B$ such that $Th\cdot x=z$ (note
that $g\cdot h=f$ can be shown as a consequence). Such a factorization
$w=Tk\cdot x$ is called a generic factorization, and thus when this
is true for all $X$, we say $T$ \emph{admits generic factorizations
\cite{WeberGeneric}.}

There are a number of natural examples of familial functors (or equivalently
functors which admit generic factorizations), with the author's favorite
being composition of spans in a category $\mathcal{E}$ with pullbacks.
\begin{example}
Given a category $\mathcal{E}$ with pullbacks, one may form the bicategory
of spans in $\mathcal{E}$, typically denoted $\mathbf{Span}\left(\mathcal{E}\right)$.
For any triple of objects $X,Y,Z\in\mathcal{E}$ the composition functor
\[
c_{X,Y,Z}\colon\mathbf{Span}\left(\mathcal{E}\right)\left(Y,Z\right)\times\mathbf{Span}\left(\mathcal{E}\right)\left(X,Y\right)\to\mathbf{Span}\left(\mathcal{E}\right)\left(X,Z\right)
\]
is familial since for any three spans $\left(s,t\right)\colon X\to Z$,
$\left(a,b\right)\colon X\to Y$ and $\left(c,d\right)\colon Y\to Z$
the universal property of the limiting cone defining the composite
of spans $\left(c,d\right)\circ\left(a,b\right)$ is a bijection $p\mapsto\left(x,h,y\right)$
as below
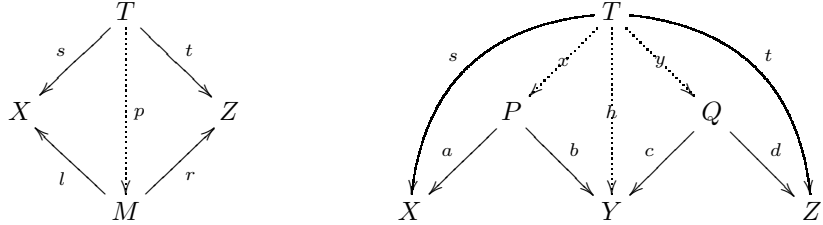
\[
\xymatrix{ & T\ar[rd]^{t}\ar[ld]_{s}\ar@{..>}[dd]^{p} &  &  &  &  & T\ar@{..>}[dd]|-{h}\ar@{..>}[dr]|-{y}\ar@{..>}[dl]|-{x}\ar@/^{2pc}/[rrdd]^{t}\ar@/_{2pc}/[lldd]_{s}\\
X &  & Z &  &  & P\ar[rd]^{b}\ar[ld]_{a} &  & Q\ar[rd]^{d}\ar[ld]_{c}\\
 & M\ar[ru]_{r}\ar[lu]^{l} &  &  & X &  & Y &  & Z
}
\]
where $\left(l,r\right)$ is the composite $\left(c,d\right)\circ\left(a,b\right)$.
Written another way, this is a natural bijection between $\mathbf{Span}\left(\mathcal{E}\right)\left(X,Z\right)\left[\left(s,t\right),\left(c,d\right)\circ\left(a,b\right)\right]$
and 
\[
\sum_{h\colon T\to Y}\mathbf{Span}\left(\mathcal{E}\right)\left(X,Y\right)\left[\left(s,h\right),\left(a,b\right)\right]\times\mathbf{Span}\left(\mathcal{E}\right)\left(Y,Z\right)\left[\left(h,t\right),\left(c,d\right)\right]
\]
and thus we directly exhibit each presheaf
\[
\mathbf{Span}\left(\mathcal{E}\right)\left(X,Z\right)\left[\left(s,t\right),-\circ-\right]\colon\mathbf{Span}\left(\mathcal{E}\right)\left(Y,Z\right)\times\mathbf{Span}\left(\mathcal{E}\right)\left(X,Y\right)\to\mathbf{Set}
\]
as a coproduct of representables, and thus exhibit $c_{X,Y,Z}$ as
a familial functor. One thing to notice here is that $c_{X,Y,Z}$
is an example of a familial functor where the domain category does
not have a terminal object; thus definitions of higher analogues of
familial functors should also not require terminal objects.
\end{example}

\subsection{The problem with pseudo familial representability}

It is the purpose of this paper to generalize these notions of familial
representability to the two-dimensional setting, replacing the category
$\mathcal{A}$ with a bicategory $\mathscr{A}$, and replacing $\mathbf{Set}$
with $\mathbf{Cat}$. However, this is not a straightforward generalization,
as asking that a pseudofunctor $F\colon\mathscr{A}\to\mathbf{Cat}$
be a coproduct of representables is often too strong of a condition.
To see why, consider the case where a pseudofunctor $T\colon\mathscr{A}\to\mathscr{B}$
is such that each $\mathscr{B}\left(X,T-\right)$ is a coproduct of
representables, meaning we have an equivalence 
\[
\mathscr{B}\left(X,T-\right)\simeq\sum_{m\in\mathfrak{M}_{X}}\mathscr{A}\left(P_{m},-\right)
\]
for some set $\mathfrak{M}_{X}$ and function $P_{\left(-\right)}\colon\mathfrak{M}_{X}\to\mathscr{A}$.
Such an equivalence must be defined by an assignation as below, which
would send each 2-cell $\alpha$ as on the left 
\[
\xymatrix@R=1em{\\
X\ar@/^{1.4pc}/[rr]^{f}\ar@/_{1.4pc}/[rr]_{g} & \ar@{}[]|-{\Downarrow\alpha} & TA & \mapsto & m, & P_{m}\ar@/^{1.4pc}/[rr]^{\overline{f}}\ar@/_{1.4pc}/[rr]_{\overline{g}} & \ar@{}[]|-{\Downarrow\overline{\alpha}} & A\\
\\
}
\]
to an $\overline{\alpha}\colon\overline{f}\Rightarrow\overline{g}$
as on the right, where $f\cong T\overline{f}\cdot\delta$ and $g\cong T\overline{g}\cdot\delta$
for the same generic $\delta\colon X\to TP_{m}$ corresponding to
the identity at $P_{m}$. This is an unreasonably strong condition:
we should not in general expect two 1-cells to factor through the
same generic just because there is a comparison map between them.\footnote{Here ``generic'' means a morphism $\delta$ corresponding to the
identity at some $P_{m}$, and the ``generic'' factorization is
obtained from substituting an $f\colon X\to TA$ into the equivalence
and applying naturality with respect to the induced $\overline{f}\colon P_{m}\to A$.} In general, this should only be expected when the comparison map
is invertible. We will therefore need a weaker notion of familial
representability in two dimensions.

\subsection{Lax familial representability}

To address the above problem, we weaken the condition on $\mathscr{B}\left(X,T-\right)$,
now only asking that it be a \emph{lax conical colimit of representables.}\footnote{A lax colimit is a weighted colimit in which the universal property
of weighted colimits replaces pseudo natural transformations with
lax natural transformations. A lax \emph{conical} colimit is such
a lax colimit where the weight $J=\Delta\mathbf{1}$ is constant at
the terminal category. We recall this notion in more detail in Subsection
\ref{laxconsec}.}\emph{ }We will use the convenient notation
\[
\mathscr{B}\left(X,T-\right)\simeq\int_{\textnormal{lax}}^{m\in\mathfrak{M}_{X}}\mathscr{A}\left(P_{m},-\right)
\]
which is justified as a lax conical colimit can be seen as an instance
of a lax coend. We then define a pseudofunctor of bicategories $T\colon\mathscr{A}\to\mathscr{B}$
to be \emph{lax familial} when each $\mathscr{B}\left(X,T-\right)$
is a lax conical colimit of representables (in a way which is natural
in $X$ in a suitable sense).

To see why being lax familial is a natural condition on a pseudofunctor
$T\colon\mathscr{A}\to\mathscr{B}$, consider the problem of calculating
a left Kan extension as below 
\[
\xymatrix@R=1em{\left[\mathscr{A}^{\textnormal{op}},\mathbf{Cat}\right]\ar@{..>}[rr]^{\textnormal{lan}_{T}} &  & \left[\mathscr{B}^{\textnormal{op}},\mathbf{Cat}\right]\\
\\
\mathscr{A}\ar[uu]^{y_{\mathscr{A}}}\ar[rr]_{T} &  & \mathscr{B}\ar[uu]_{y_{\mathscr{B}}}
}
\]
for a given pseudofunctor $T$ (where $\mathscr{A}$ and $\mathscr{B}$
are small). In general this left extension should not be expected
to have a nice form. However, if $T$ is a pseudofunctor that is lax
familial, so that each $\mathscr{B}\left(X,T-\right)$ is a lax conical
colimit of representables, then this left extension will have a simpler
description. Said in more detail, such a left extension is generally
computed as a bi-coend (whose construction generally requires formally
adding in isomorphisms, hence the complexity), but in the case where
$T$ is lax familial the left extension may be more easily computed
as a lax coend.

An important example of this situation (shown to be lax familial in
Example \ref{spanembfam}) is given by taking $T$ as the canonical
inclusion of a small category $\mathcal{E}$ with pullbacks into its
bicategory of spans $\mathbf{Span}\left(\mathcal{E}\right)$ 
\[
\xymatrix@R=1em{\left[\mathcal{E}^{\textnormal{op}},\mathbf{Cat}\right]\ar@{..>}[rr]^{\textnormal{lan}_{T}\quad\;\;\;} &  & \left[\mathscr{\mathbf{Span}\left(\mathcal{E}\right)}^{\textnormal{op}},\mathbf{Cat}\right]\\
\\
\mathcal{E}\ar[uu]^{y_{\mathcal{E}}}\ar[rr]_{T} &  & \mathbf{Span}\left(\mathcal{E}\right)\ar[uu]_{y_{\mathbf{Span}\left(\mathcal{E}\right)}}
}
\]
and forming the left extension $\textnormal{lan}_{T}$ as above, with
right adjoint $\textnormal{res}_{T}$ given by restricting along $T$.
Now, recognizing $\left[\mathscr{\mathbf{Span}\left(\mathcal{E}\right)}^{\textnormal{op}},\mathbf{Cat}\right]$
as the 2-category of fibrations with sums (by the universal property
of spans) \cite{unispans}, and noting that the extension-restriction
adjunction is pseudomonadic (a consequence of $T$ being bijective
on objects) \cite{psmonadicity}, the reader will recognize this left
extension as the free functor for the pseudomonad $\Sigma_{\mathcal{E}}$
for fibrations over $\mathcal{E}$ with sums. In this way one can
derive the pseudomonad for fibrations with sums (as the composite
$\textnormal{res}_{T}\cdot\textnormal{lan}_{T}$), and understand
why this pseudomonad has a simpler description that one would generally
expect for pseudomonads arising in this way. Note the same can be
done for fibrations with products, replacing $\mathbf{Span}\left(\mathcal{E}\right)$
with $\mathbf{Span}\left(\mathcal{E}\right)^{\textnormal{co}}$.

Also note that it will be shown in future work that the above is a
special case of a more general result; for a category\footnote{More generally, one can use a bicategory here. However, it suffices
to use a 1-category in most of the interesting examples.} $\mathcal{E}$ and bicategory $\mathscr{B}$, bijective on objects
pseudofunctors $\mathcal{E}\to\mathscr{B}$ correspond with bi-cocontinuous
pseudomonads on $\left[\mathcal{E}^{\textnormal{op}},\mathbf{Cat}\right]$,
in which case the pseudo-algebras of the pseudomonad form the 2-category
$\left[\mathscr{B}^{\textnormal{op}},\mathbf{Cat}\right]$. 

Of course, whilst understanding the above situation is the author's
motivation, there are other motivating examples of lax familial functors.
For instance the results of Weber \cite{Webfam} (and the later shown
equivalence with our definition) show that the composite 2-monads
of \cite{Weber2005} that describe symmetric and braided analogues
of the \textgreek{w}-operads of \cite{Bat1998} are examples of lax
familial functors.

\subsection{Main results and structure of the paper}

The goal of this paper is to generalize all of the important results
true for familial functors to the case of lax familial functors. The
notions given in this paper concerning lax familial representability
and its corresponding notions of genericity are all new. The only
exception is to note that Weber's definitions of familial (2-)functors
\cite{Webfam} and notions of genericity can be seen as instances
of our definitions in the case where we have a terminal object (though
this is not obvious, as Weber's definitions use fibrations and look
considerably different at first glance).

In Section \ref{genelements} we generalize the basic result that
a presheaf is a coproduct of representables if and only if each connected
component of its category of elements has an initial object, now giving
a description of when a $\mathbf{Cat}$-valued presheaf is a lax conical
colimit of representables in Theorem \ref{p5:laxequiv}.

In Section \ref{laxgenfacsec} we generalize the result that a functor
is familial if and only if it admits generic factorizations, by defining
a notion of lax-generic factorizations and showing this condition
is equivalent to being lax familial in Theorem \ref{laxequivT}. Note
that such an equivalence was not shown by Weber \cite{Webfam}.

The main result of Section \ref{webersection} is Theorem \ref{webermain},
which states that (a slightly stricter version of) our definition
of lax familial pseudofunctors is equivalent to Weber's very different-looking
definition. This is strong evidence that our definitions of familial
pseudofunctors are the correct ones.

In Section \ref{examples} we provide a number of natural examples
of lax familial pseudofunctors, further justifying our definitions. 

In Section \ref{spectrumsec} we generalize Diers' result \cite{Diers}
that a functor is familial if and only if it factors as a right adjoint
followed by a discrete fibration. This generalization is given in
Theorem \ref{specfacthm}, which states that a pseudofunctor is lax
familial if and only if it factors as an appropriate right lax $\mathsf{F}$-adjoint
(a special type of lax adjunction), followed by the two-dimensional
version of a discrete fibration \cite{2discfib}. In the author's
opinion this gives the most natural-looking characterization of lax
familial pseudofunctors, as the other characterizations appear quite
technical.

\section{Background}

In this section we will recall the necessary background knowledge
for this paper. We will first recall the basic theory of familial
functors and generic factorizations in the one-dimensional case \cite{WeberGeneric},
and then go on to recall the basics of lax conical colimits \cite{Street2Limits}
and the Grothendieck construction \cite{fibbicaty}, which will replace
the category of elements in the two-dimensional setting.

\subsection{Generic factorizations in one dimension}

We will first recall the basic fact that a presheaf is a coproduct
of representables if and only if each connected component of its category
of elements has an initial object. It is worth explaining this result
in some more detail, as later on in the two-dimensional case simple
conditions such as asking each connected component has an initial
object will not suffice. Of course these are all well-known results
of Diers \cite{Diers,DiersDiag} (also see \cite{WeberGeneric} for
a more recent account).
\begin{defn}
Given a presheaf $F\colon\mathcal{A}\to\mathbf{Set}$, recall the
classical notion of the category of elements of $F$ as the category
with objects given by pairs $\left(A\in\mathcal{A},x\in FA\right)$
and morphisms $\left(A,x\right)\nrightarrow\left(B,y\right)$ given
by maps $f\colon A\to B$ such that $Ff\left(x\right)=y$. We denote
this category $\textnormal{el }F$. 
\end{defn}

The following definition is more complicated than it needs to be,
in that the generics are precisely the initial objects in the connected
components. The reason for stating it this way is that it more closely
matches the definitions needed in two dimensions.
\begin{defn}
\label{defgeneric1} Given a presheaf $F\colon\mathcal{A}\to\mathbf{Set}$,
we say an object $\left(A,x\right)\in\textnormal{el }F$ is \emph{el-generic
}if for any given objects $\left(B,y\right)$, $\left(C,z\right)$
and morphisms $f$ and $g$ as below 
\[
\xymatrix{\; & \left(C,z\right)\myar{g}{d}\\
\left(A,x\right)\myard{f}{r}\ar@{..>}[ur]^{h} & \left(B,y\right)
}
\]
there exists a unique morphism $h\colon\left(A,x\right)\to\left(C,z\right)$.
It is then automatic that the above triangle commutes.

Moreover, given two el-generic objects $\left(A,x\right)$ and $\left(D,w\right)$,
an \emph{el-generic morphism} $\left(A,x\right)\to\left(D,w\right)$
is any such morphism in $\textnormal{el }F$.\footnote{It is trivial to check that any such el-generic morphism is both invertible
and unique.}
\end{defn}

\begin{rem}
The reader will note that this is stronger than asking for the existence
of a unique lifting $h$. In fact, asking that $h$ be the unique
morphism (and not just the unique lifting), is a condition which will
turn out often to be too strong in dimension two. Nevertheless, this
is the correct definition for dimension one, when one requires the
indexing\footnote{For a presheaf $F$, the ``indexing'' $\mathfrak{M}^{F}$ refers
to the category of el-generics and el-generic morphisms between them.} to be a set.

If $h$ was only required to be the unique lifting, our indexing would
only be a groupoid in general, as any el-generic morphism $\left(A,x\right)\to\left(D,w\right)$
will still be invertible, but perhaps not unique. This weaker version
of el-generic objects arises in the study of stable functors under
the name of ``candidate'' generics \cite{TaylorCandidate}, and
also appears in the study of qualitative domains \cite{Lam88,Gir86}.
It is interesting to note that our two-dimensional versions of famillial
representability will restrict to this weaker version in dimension
one.
\end{rem}

The basic result describing when a presheaf is a coproduct of representables
is then the following.
\begin{prop}[Diers \cite{Diers}]
\label{coprodrep} Given a presheaf $F\colon\mathcal{A}\to\mathbf{Set}$,
the following are equivalent: 
\begin{enumerate}
\item $F\colon\mathcal{A}\to\mathbf{Set}$ is a coproduct of representables; 
\item each connected component of $\textnormal{el }F$ has an initial object; 
\item for any $\left(B,y\right)\in\textnormal{el }F$ there exists a generic
object $\left(A,x\right)$ and morphism $f\colon\left(A,x\right)\to\left(B,y\right)$.\footnote{Clearly in this situation the generic object is necessarily unique.}
\end{enumerate}
\end{prop}

\begin{rem}
Of course (3) above is simply expanding (2) into more detail. This
more detailed version will be more analogous to the characterizations
we give in the two-dimensional case.
\end{rem}

We now consider the case of a functor $T\colon\mathcal{A}\to\mathcal{B}$,
asking when $\mathcal{B}\left(X,T-\right)$ is a coproduct of representables.
\begin{defn}
\label{leftmultiadjoint} We say a functor $T\colon\mathcal{A}\to\mathcal{B}$
is \emph{familial} if for every $X\in\mathcal{B}$ the presheaf $\mathcal{B}\left(X,T-\right)\colon\mathcal{A}\to\mathbf{Set}$
is a coproduct of representables. 
\end{defn}

Specializing Definition \ref{defgeneric1} to this case, we recover
the following definition of a ``generic morphism'' (also known as
``diagonally universal morphism'' in the work of Diers). This definition
is originally due to Diers \cite{Diers}, but we follow the terminology
of Weber \cite{WeberGeneric} and \cite{CarbJohnFam}.
\begin{defn}
Given a functor $T\colon\mathcal{A}\to\mathcal{B}$ we say that a
morphism $x\colon X\to TA$ for some $X\in\mathcal{B}$ and $A\in\mathcal{A}$
is \emph{generic} if for any commuting square as on the left below
\[
\xymatrix{X\myar{z}{r}\ar[d]_{x} & TB\ar[d]^{Tg} &  &  & X\myar{z}{r}\ar[d]_{x} & TB\ar[d]^{Tg}\\
TA\myard{Tf}{r} & TC & \; &  & TA\myard{Tf}{r}\ar[ur]|-{Th} & TC
}
\]
there exists a unique $h\colon A\to B$ such that $Th\cdot x=z$.
That $f=g\cdot h$ follows as a consequence of this property. 
\end{defn}

Applying Proposition \ref{coprodrep} to presheaves of the form $\mathcal{B}\left(X,T-\right)$
for a given functor $T\colon\mathcal{A}\to\mathcal{B}$, we obtain
the following characterization of familial functors in terms of these
generic morphisms.
\begin{prop}[Diers \cite{Diers}]
Given a functor $T\colon\mathcal{A}\to\mathcal{B}$, the following
are equivalent: 
\begin{enumerate}
\item the functor $T$ is familial; 
\item for every morphism $f\colon X\to TW$ there exists a generic morphism
$\delta\colon X\to TA$ and morphism $\overline{f}\colon A\to W$
such that $f=T\overline{f}\cdot\delta$. 
\end{enumerate}
\end{prop}

Following Weber's terminology, condition (2) is often stated another
way.
\begin{defn}
\cite{WeberGeneric} Given a functor $T\colon\mathcal{A}\to\mathcal{B}$,
if condition (2) above is satisfied, we say that \emph{$T$ admits
generic factorizations}. 
\end{defn}

\subsection{Lax conical colimits and the Grothendieck construction\label{laxconsec}}

Here we give the required background on lax conical colimits and the
Grothendieck construction. In our convention, we specify the direction
of 2-cells in a lax natural transformation $\alpha\colon F\Rightarrow G$
as below
\[
\xymatrix{FA\ar[r]^{Ff}\ar[d]_{\alpha_{A}}\utwocell[0.5]{dr}{\alpha_{f}} & FB\ar[d]^{\alpha_{B}}\\
GA\ar[r]_{Gf} & GB
}
\]
for any morphism $f\colon A\to B$ in the domain bicategory, with
2-cells in the opposite direction defining oplax natural transformations.
\begin{defn}[lax conical colimits \cite{Street2Limits}]
\label{laxconcolim} Given a category $\mathcal{A}$, a bicategory
$\mathscr{K}$, and pseudofunctor $F\colon\mathcal{A}\to\mathscr{K}$,
the \emph{lax conical colimit} of $F$ consists of an object $T\in\mathscr{K}$,
along with for every $A\in\mathcal{A}$ a map $\varphi_{A}\colon FA\to T$
and for every morphism $f\colon A\to B$ in $\mathcal{A}$ a 2-cell
\[
\xymatrix{ & T\\
FA\ar[ur]^{\varphi_{A}}\ar[rr]_{Ff} & \ltwocell[0.4]{u}{\varphi_{f}} & FB\ar[ul]_{\varphi_{B}}
}
\]
compatible with the binary and nullary constraints of $F$ \cite{Street2Limits}.
This data, which is also called a lax cocone, and may be seen as a
lax natural transformation $\varphi\colon\Delta\mathbf{1}\Rightarrow\mathscr{K}\left(F-,T\right)\colon\mathcal{A}^{\textnormal{op}}\to\mathbf{Cat}$,
is required to be universal in that 
\[
\begin{aligned}\mathscr{K}\left(T,S\right) & \to\left[\mathcal{A}^{\textnormal{op}},\mathbf{Cat}\right]\left(\Delta\mathbf{1},\mathscr{K}\left(F-,S\right)\right)\\
\alpha & \mapsto\mathscr{K}\left(F-,\alpha\right)\cdot\varphi
\end{aligned}
\]
defines an equivalence (where $\left[\mathcal{A}^{\textnormal{op}},\mathbf{Cat}\right]$
is the 2-category of pseudofunctors, lax natural transformations,
and modifications). If one reverses the direction of the 2-cell $\varphi_{f}$,
and replaces lax transformations with oplax transformations, one then
has the notion of an \emph{oplax colimit}. The name \emph{conical}
refers to the fact that the weight of such a colimit is the terminal
presheaf $\Delta\mathbf{1}$, so that the data takes the form of cone-shaped
diagrams.
\end{defn}

\begin{rem}
It is worth noting that the above definition can be used when $F\colon\mathcal{A}\to\mathscr{K}$
is only required to be a lax functor. Also, one may note that lax
conical colimits can be seen as an instance of weighted bi-colimits
(though we will not use this). 
\end{rem}

When $\mathscr{K}=\mathbf{Cat}$, such a lax colimit can easily be
calculated using the Grothendieck construction. We will proceed to
describe this construction below (though we will be more general by
replacing the category $\mathcal{A}$ with a bicategory $\mathscr{A}$). 
\begin{defn}[Grothendieck construction]
\label{groth} Given a bicategory $\mathscr{A}$ and pseudofunctor
$F\colon\mathscr{A}\to\mathbf{Cat}$, the \emph{(bi)category of elements}
of $F$, denoted by $\textnormal{el }F$ or by 
\[
\int_{\textnormal{lax}}^{A\in\mathscr{A}}FA
\]
is the bicategory with: 
\begin{description}
\item [{Objects}] An object is a pair of the form $\left(A\in\mathscr{A},x\in FA\right)$; 
\item [{Morphisms}] A morphism $\left(A,x\right)\nrightarrow\left(B,y\right)$
is a pair $f\colon A\to B$ in $\mathscr{A}$ and $\alpha\colon Ff\left(x\right)\to y$
in $FB$; we say such a morphism $\left(f,\alpha\right)$ is \emph{opcartesian}
if $\alpha$ is invertible;
\item [{2-cells}] A 2-cell $\left(f,\alpha\right)\Rightarrow\left(g,\beta\right)\colon\left(A,x\right)\nrightarrow\left(B,y\right)$
is a 2-cell $\nu\colon f\Rightarrow g$ in $\mathscr{A}$ rendering
commutative
\[
\xymatrix{Ff\left(x\right)\myar{\left(F\nu\right)_{x}}{r}\ar@/_{1pc}/[rr]_{\alpha} & Fg\left(x\right)\myar{\beta}{r} & y}
.
\]
\end{description}
It is also common to refer to the bicategory $\int_{\textnormal{lax}}^{A\in\mathscr{A}}FA$
with its canonical projection to $\mathscr{A}$ as the \emph{Grothendieck
construction} of $F$. 
\end{defn}

The ``lax coend'' notation as used above is defined as follows.
However, we will not burden this paper with all the technical details
of the definition.
\begin{defn}
Let $\mathscr{A}$ be a bicategory and let $F\colon\mathscr{A}\to\mathbf{Cat}$
and $G\colon\mathscr{A}^{\textnormal{op}}\to\mathbf{Cat}$ be pseudofunctors.
We define the \emph{lax-coend}
\[
\int_{\textnormal{lax}}^{A\in\mathscr{A}}FA\times GA
\]
as the vertex of the universal diagram, over each morphism $f\colon A\to B$
in $\mathscr{A}$, 
\[
\xymatrix@R=0.8em{ & FA\times GB\ar[rd]^{Ff\times\textnormal{id}}\ar[ld]_{\textnormal{id}\times Gf}\\
FA\times GA\myard{\epsilon_{A}}{rd}\rtwocell{rr}{\sigma} &  & FB\times GB\myar{\epsilon_{B}}{ld}\\
 & \int_{\textnormal{lax}}^{A\in\mathscr{A}}FA\times GA
}
\]
subject to the canonical nullary and binary coherence conditions with
regards to the pseudofunctoriality of $F$ and $G$ \cite{Foscocoend}.
\end{defn}

\begin{rem}
When $\mathscr{A}$ is a category, the notation $\int_{\textnormal{lax}}^{A\in\mathscr{A}}FA$
is justified, as the category of elements can be written as a lax
colimit as in Definition \ref{laxconcolim}. In the case where $\mathscr{A}$
is a bicategory, $\textnormal{el }F$ is an appropriate tri-colimit
of $F$, and the notation is still justified (though in a more technical
sense that we will not burden this paper with; see \cite{Buckley}). 
\end{rem}

\begin{rem}
To make clear the duality of covariance and contravariance in the
above construction we note the following. For a pseudofunctor $F\colon\mathscr{A}\to\mathbf{Cat}$,
its lax colimit is given by the lax coend $\int_{\textnormal{lax}}^{A\in\mathscr{A}}FA$.
In the contravariant case of a pseudofunctor $G\colon\mathscr{A}^{\textnormal{op}}\to\mathbf{Cat}$,
the lax coend $\int_{\textnormal{lax}}^{A\in\mathscr{A}}GA$ coincides
with the oplax colimit of $G$, which could also be written as the
oplax coend $\int_{\textnormal{oplax}}^{A\in\mathscr{A^{\textnormal{op}}}}GA$.
Typically given a contravariant $G\colon\mathscr{A}^{\textnormal{op}}\to\mathbf{Cat}$,
one takes the oplax colimit of $G$ as its category of elements (so
that we may project from this category of elements into $\mathscr{A}$).
\end{rem}

Finally, we should recall the notion of a fibration and cartesian
morphisms.
\begin{defn}
\cite{fibbicaty} Let $p\colon\mathcal{F}\to\mathcal{E}$ be a functor.
We say a morphism $\phi\colon W\to B$ in $\mathcal{F}$ is \emph{$p$-cartesian}
if for any $\psi\colon A\to B$ and $r\colon pA\to pW$ such that
the right diagram below commutes
\[
\xymatrix{W\ar[r]^{\phi} & B &  &  & pW\ar[r]^{p\phi} & pB\\
A\ar[ur]_{\psi}\ar@{..>}[u]^{\overline{r}} &  &  &  & pA\ar[ur]_{p\psi}\ar[u]^{r}
}
\]
there exists a unique $\overline{r}\colon A\to W$ such that $p\overline{r}=r$
and the left diagram above commutes. If for any morphism $f\colon X\to pB$
in $\mathcal{E}$ there exists a $p$-cartesian morphism $\phi\colon f^{*}B\to B$
in $\mathcal{F}$ such that $p\left(\phi\right)=f$, we then say $p$
is a \emph{fibration}.
\end{defn}

\section{Lax generics in bicategories of elements\label{genelements}}

Generalizing the fact that a presheaf is a coproduct of representables
if and only if each connected component of the category of elements
has an initial object, our first main goal is to understand when a
pseudofunctor $F\colon\mathscr{A}\to\mathbf{Cat}$ a lax conical colimit
of representables, written
\[
F\simeq\int_{\textnormal{lax}}^{m\in\mathfrak{M}}\mathscr{A}\left(P_{m},-\right)
\]
for some $\mathfrak{M\in\mathbf{Cat}}$ and pseudofunctor $P_{\left(-\right)}\colon\mathfrak{M}\to\mathscr{A}$,
giving a characterization in terms of the category of elements of
$F$ (which is analogous to each connected component having an initial
object in the one dimensional case). However, before we can describe
lax el-generic\footnote{We have the ``el'' here as this notion of genericity is used in
bicategories of elements.} objects (the appropriate analogue of these initial objects) and morphisms
in bicategories of elements, we will have to introduce the language
needed to describe them. In particular, we introduce ``mixed left
liftings'' which are similar to left liftings \cite{companion},
except that the induced arrow's direction is reversed. Note that basic
properties for left liftings, such as the pasting lemma, or the lifting
through an identity being itself, do not hold in general for mixed
left liftings \cite{yonedastructures}. 
\begin{defn}[mixed left lifting property]
Let $\mathscr{C}$ be a bicategory. We say a diagram as on the left
below 
\[
\xymatrix{\; & \mathcal{C}\myar{g}{d} &  &  & \; & \mathcal{C}\myar{g}{d}\\
\mathcal{A}\myard{f}{r}\ar@{..>}[ur]^{h} & \mathcal{B}\utwocell[0.35]{lu}{\nu} &  &  & \mathcal{A}\myard{f}{r}\ar[ur]^{k} & \mathcal{B}\utwocell[0.35]{lu}{\psi}
}
\]
exhibits $\left(h,\nu\right)$ as the \emph{mixed left lifting }of
$f$ through $g$ if for any diagram as on the right above, there
exists a unique 2-cell $\lambda\colon k\Rightarrow h$ such that

\[
\xymatrix{\; & \mathcal{C}\myar{g}{d} & \ar@{}[d]|-{=} & \; & \mathcal{C}\myar{g}{d}\\
\mathcal{A}\myard{f}{r}\ar[ur]|-{k}\ar@/^{2pc}/[ur]^{h} & \mathcal{B}\utwocell[0.35]{lu}{\psi}\utwocell[0.75]{lu}{\lambda} & \; & \mathcal{A}\myard{f}{r}\ar[ur]^{h} & \mathcal{B}\utwocell[0.35]{lu}{\nu}\mathrlap{\ .}
}
\]
Moreover, we say such a lifting $\left(h,\nu\right)$ is \emph{strong
}if $h$ is subterminal\footnote{A subterminal object $I$ in a category is one where any morphism
into it is unique if it exists. When the category in question has
a terminal object, this is equivalent to the unique morphism $I\to\mathbf{1}$
being a monomorphism; hence the name subterminal.} in $\mathscr{C}\left(\mathcal{A},\mathcal{C}\right)$. 
\end{defn}

\begin{rem}
It is not hard to see that strong mixed liftings are unique up to
unique isomorphism. Indeed, it is this stronger notion that will be
used though this section. 
\end{rem}

The following lemma shows that an arrow $h$ which arises as a strong
mixed lifting has the property that the strong mixed lifting of $h$
through the identity is itself.
\begin{lem}
\label{inducedgen} Suppose the left diagram below 
\[
\xymatrix{\; & \mathcal{C}\myar{g}{d} &  &  & \; & \mathcal{C}\myar{1_{\mathcal{C}}}{d}\\
\mathcal{A}\myard{f}{r}\ar[ur]^{h} & \mathcal{B}\utwocell[0.35]{lu}{\nu} &  &  & \mathcal{A}\myard{h}{r}\ar[ur]^{h} & \mathcal{C}\utwocell[0.37]{lu}{\textnormal{id}}
}
\]
exhibits $\left(h,\nu\right)$ as the strong mixed lifting of $f$
through $g$. Then the right diagram above exhibits $\left(h,\textnormal{id}\right)$
as the strong mixed lifting of $h$ through $1_{\mathcal{C}}.$ 
\end{lem}

\begin{proof}
Given any $k\colon\mathcal{A}\to\mathcal{C}$ and $\zeta\colon h\Rightarrow k$
we have by universality of $\left(h,\nu\right)$ an induced $\lambda\colon k\Rightarrow h$
such that

\[
\xymatrix{\; & \mathcal{C}\myar{g}{d} & \ar@{}[d]|-{=} & \; & \mathcal{C}\myar{g}{d}\\
\mathcal{A}\myard{f}{r}\ar[ur]|-{h}\ar@/^{1.3pc}/[ur]|-{k}\ar@/^{2.6pc}/[ur]^{h} & \mathcal{B}\utwocell[0.35]{lu}{\nu}\ultwocell[0.65]{lu}{\zeta}\ultwocell[0.96]{lu}{\lambda} & \; & \mathcal{A}\myard{f}{r}\ar[ur]^{h} & \mathcal{B}\utwocell[0.35]{lu}{\nu}\mathrlap{\ ;}
}
\]
that is, since $h$ is subterminal, a unique induced $\lambda\colon k\Rightarrow h$
such that $\lambda\zeta$ is the identity. This proves the result. 
\end{proof}
We now have the necessary background to introduce the notions of lax
el-generic object and el-generic morphism in bicategories of elements. 
\begin{defn}[lax el-generic objects]
\label{deflaxgenericobject} Let $\mathscr{A}$ be a bicategory and
$F\colon\mathscr{A}\to\mathbf{Cat}$ be a pseudofunctor. We say that
an object $\left(A,x\right)$ in $\textnormal{el }F$ is a \emph{lax
el-generic object} if for any $\left(B,y\right)$, $\left(C,z\right)$,
$\left(f,\alpha\right)$ and $\left(g,\beta\right)$ as below with
$\beta$ invertible 
\[
\xymatrix{\; & \left(C,z\right)\myar{\left(g,\beta\right)}{d}\\
\left(A,x\right)\myard{\left(f,\alpha\right)}{r}\ar@{..>}[ur]^{\left(h,\gamma\right)} & \left(B,y\right)\utwocell[0.35]{lu}{\nu}
}
\]
\begin{enumerate}
\item there exists a strong mixed left lifting $\left(h,\gamma\right)\colon\left(A,x\right)\to\left(C,z\right)$
exhibited by a 2-cell $\nu\colon f\Rightarrow gh$; 
\item if $\alpha$ is invertible above, then both $\gamma$ and $\nu$ are
also invertible. 
\end{enumerate}
\end{defn}

\begin{rem}
If we replace the isomorphism $\beta$ with an identity above, the
definition remains equivalent. 
\end{rem}

\begin{defn}[el-generic morphisms]
\label{defgenericmorphism} Let $\mathscr{A}$ be a bicategory and
$F\colon\mathscr{A}\to\mathbf{Cat}$ be a pseudofunctor, and suppose
that $\left(A,x\right)$ is a lax-generic object in $\textnormal{el }F$.
We say that a morphism $\left(\ell,\phi\right)\colon\left(A,x\right)\rightarrow\left(D,w\right)$
out of $\left(A,x\right)$ in $\textnormal{el }F$ is an \emph{el-generic
morphism} if the diagram below 
\[
\xymatrix{\; & \left(D,w\right)\myar{\left(1_{D},\textnormal{id}\right)}{d}\\
\left(A,x\right)\myard{\left(\ell,\phi\right)}{r}\ar@{..>}[ur]^{\left(\ell,\phi\right)} & \left(D,w\right)\utwocell[0.35]{lu}{\textnormal{id}}
}
\]
exhibits $\left(\ell,\phi\right)$ as the strong mixed left lifting
of $\left(\ell,\phi\right)$ through $\left(1_{D},\textnormal{id}\right)$. 
\end{defn}

\begin{rem}
It is clear from the second part of Definition \ref{deflaxgenericobject}
that for any opcartesian $\left(\ell,\phi\right)$ (meaning $\phi$
is invertible) out of a lax el-generic $\left(A,x\right)$, the induced
mixed lifting must be (isomorphic to) $\left(\ell,\phi\right)$. Thus
opcartesian morphisms are always el-generic.
\end{rem}

\begin{rem}
It is an easy consequence of the universal property that every 2-cell
out of $\left(\ell,\phi\right)$ is a section (in a unique way); and
consequently that any 2-cell between el-generic morphisms is invertible.
Moreover, as $\left(\ell,\phi\right)$ is subterminal within its hom-category
it follows that any isomorphism between el-generic morphisms is unique.
It follows that if $\left(A,x\right)$ and $\left(B,y\right)$ are
lax el-generic objects, then the category of el-generic morphisms
$\left(A,x\right)\to\left(B,y\right)$ is equivalent to a discrete
category (a set). 
\end{rem}

\begin{rem}
It is worth noting that for any lax el-generic object $\left(A,x\right)$
and strong mixed lifting as below 
\[
\xymatrix{\; & \left(C,z\right)\myar{\left(g,\beta\right)}{d}\\
\left(A,x\right)\myard{\left(f,\alpha\right)}{r}\ar@{..>}[ur]^{\left(h,\gamma\right)} & \left(B,y\right)\utwocell[0.35]{lu}{\nu}
}
\]
with $\beta$ invertible, the induced morphism $\left(h,\gamma\right)$
is an el-generic morphism as a consequence of Lemma \ref{inducedgen}. 
\end{rem}

The following is a step towards characterizing when an $F\colon\mathscr{A}\to\mathbf{Cat}$
is a lax conical colimit of representables, indexed by the following
category of el-generic objects and morphisms. The reader will note
the importance of the indexing being a 1-category (and thus the consideration
of representatives of equivalence classes of morphisms in order to
obtain a 1-category), just as it is important in the one-dimensional
case that the indexing is a set.
\begin{defn}
Let $\mathscr{A}$ be a bicategory and let $F\colon\mathscr{A}\to\mathbf{Cat}$
be a pseudofunctor. Suppose that el-generic morphisms between lax
el-generic objects compose to el-generic morphisms. Define $\mathscr{A}_{g}^{F}$
as the sub-bicategory of $\textnormal{el }F$ consisting of lax el-generic
objects and el-generic morphisms, and define $\mathfrak{M}^{F}$ as
the 1-category consisting of lax el-generic objects in $\textnormal{el }F$
and chosen representatives of isomorphism classes of el-generic morphisms.
\end{defn}

Of course, in the above definition we will have an equivalence $\mathscr{A}_{g}^{F}\simeq\mathfrak{M}^{F}$.
\begin{prop}
\label{ff} Let $\mathscr{A}$ be a bicategory and let $F\colon\mathscr{A}\to\mathbf{Cat}$
be a pseudofunctor. Suppose that el-generic morphisms between lax
el-generic objects compose to el-generic morphisms. Let $P_{\left(-\right)}\colon\mathfrak{M}\to\mathscr{A}$
be the assignment sending a lax el-generic object $\left(A,x\right)$
to $A$ and a representative el-generic morphism between el-generic
objects $\left(s,\phi\right)\colon\left(A,x\right)\to\left(B,y\right)$
to $s\colon A\to B$. Then $P_{\left(-\right)}\colon\mathfrak{M}\to\mathscr{A}$
defines a pseudofunctor, and for every $T\in\mathscr{A}$ there exists
fully faithful functors 
\[
\Lambda_{T}\colon\int_{\textnormal{lax}}^{m\in\mathfrak{M}^{F}}\mathscr{A}\left(P_{m},T\right)\to FT
\]
pseudo-natural in $T\in\mathscr{A}$. 
\end{prop}

\begin{proof}
Firstly note that $P_{\left(-\right)}\colon\mathfrak{M}\to\mathscr{A}$
defines a pseudofunctor since it may be written as the composite $\mathfrak{M}^{F}\to\mathscr{A}_{g}^{F}\to\textnormal{el }F\to\mathscr{A}$.
We may then define $\Lambda_{T}$ on objects by the assignment $\left(A,x,f\right)\mapsto Ff\left(x\right)$,
and on morphisms by the assignment (suppressing the pseudofunctoriality
constraints of $F$) 
\begin{equation}
\xymatrix{\left(A,x,f\colon A\to T\right)\ar@{~>}[dd]^{_{\left(h,\gamma,\nu\right)}} & A\ar[dd]_{h} & Fh\left(x\right)\ar[dd]_{\gamma} & A\ar[dd]_{h}\ar[rd]^{f} &  &  & Ff\left(x\right)\ar[d]^{\left(F\nu\right)_{x}}\\
 &  &  & \dtwocell[0.35]{r}{\nu} & T & \mapsto & FgFh\left(x\right)\ar[d]^{Fg\left(\gamma\right)}\\
\left(B,y,g\colon B\to T\right) & B & y & B\ar[ur]_{g} &  &  & Fg\left(y\right)\mathrlap{\ .}
}
\label{lambdadef}
\end{equation}
Observe that we have the following conditions satisfied.

\noun{Functoriality.} Given another 
\[
\xymatrix{\left(B,y,g\colon B\to T\right)\ar@{~>}[dd]^{_{\left(k,\zeta,\mu\right)}} & B\ar[dd]_{k} & Fk\left(y\right)\ar[dd]_{\zeta} & B\ar[dd]_{k}\ar[rd]^{g} &  &  & Fg\left(y\right)\ar[d]^{\left(F\mu\right)_{y}}\\
 &  &  & \dtwocell[0.35]{r}{\mu} & T & \mapsto & FqFk\left(y\right)\ar[d]^{Fq\left(\zeta\right)}\\
\left(C,z,q\colon A\to T\right) & C & z & C\ar[ur]_{q} &  &  & Fq\left(z\right)
}
\]
the commutativity of 
\[
\xymatrix{Ff\left(x\right)\myar{\left(F\nu\right)_{x}}{r} & FgFh\left(x\right)\myar{Fg\left(\gamma\right)}{r}\ar[rd]_{\left(F\mu\right)_{Fh\left(x\right)}} & Fg\left(y\right)\myar{\left(F\mu\right)_{y}}{r} & FqFk\left(y\right)\myar{Fq\left(\zeta\right)}{r} & Fq\left(z\right)\\
 &  & FqFkFh\left(x\right)\ar[ru]_{FqFk\left(\gamma\right)}
}
\]
by naturality of $F\mu$ shows $\Lambda_{T}$ respects binary composition.
It is trivial that identities are preserved.

\noun{Fullness.} Given any objects $\left(A,x,f\colon A\to T\right)$
and $\left(B,y,g\colon B\to T\right)$ and morphism $\phi\colon Ff\left(x\right)\to Fg\left(y\right)$,
we may construct the universal diagram 
\[
\xymatrix{\; & \left(B,y\right)\myar{\left(g,\textnormal{id}\right)}{d}\\
\left(A,x\right)\myard{\left(f,\phi\right)}{r}\ar@{..>}[ur]^{\left(h,\gamma\right)} & \left(B,Fg\left(y\right)\right)\utwocell[0.35]{lu}{\nu}
}
\]
using lax el-genericity of $\left(A,x\right)$. Now $\left(h,\gamma\right)$
is el-generic by Lemma \ref{inducedgen}, and without loss of generality
we can assume it is a representative el-generic (since it is necessarily
isomorphic to one). Then $\Lambda_{T}\left(h,\gamma,\nu\right)=\phi$.

\noun{Faithfulness.} Given another triple $\left(k,\psi,\omega\right)$
such that $\Lambda_{T}\left(k,\psi,\omega\right)=\phi$, we have the
diagram 
\[
\xymatrix{\; & \left(B,y\right)\myar{\left(g,\textnormal{id}\right)}{d}\\
\left(A,x\right)\myard{\left(f,\phi\right)}{r}\ar@{..>}[ur]^{\left(k,\psi\right)} & \left(B,Fg\left(y\right)\right)\utwocell[0.35]{lu}{\omega}
}
\]
But as $\left(k,\psi\right)$ and $\left(h,\gamma\right)$ are both
el-generics, the induced $\left(k,\psi\right)\Rightarrow\left(h,\gamma\right)$
arising from universality of $\left(h,\gamma\right)$ must be invertible.
Also, as they are both representative, they must be equal. As the
identity must then be the induced morphism we conclude $k=h$, $\psi=\gamma$
and $\omega=\nu$.

\noun{Pseudo-naturality.} Given any 1-cell $\alpha\colon T\to S$
in $\mathscr{A}$ the squares 
\[
\xymatrix{\left(A,x,f\colon A\to T\right)\ar[d]_{\Lambda_{T}}\myar{\alpha\cdot\left(-\right)}{r} & \left(A,x,\alpha f\colon A\to S\right)\ar[d]^{\Lambda_{S}}\\
Ff\left(x\right)\myard{F\alpha\cdot\left(-\right)}{r} & F\left(\alpha f\right)\left(x\right)
}
\]
commute up to pseudo-functoriality constraints of $F$, and the family
of squares of the form above satisfy the required naturality, nullary
and binary coherence conditions as a consequence of the corresponding
pseudo-functoriality coherence conditions. 
\end{proof}
We can now characterize precisely when a pseudofunctor $F\colon\mathscr{A}\to\mathbf{Cat}$
is a lax conical colimit of representables. 
\begin{thm}
\label{p5:laxequiv} Let $\mathscr{A}$ be a bicategory and $F\colon\mathscr{A}\to\mathbf{Cat}$
be a pseudofunctor. Then the following are equivalent: 
\begin{enumerate}
\item the pseudofunctor $F\colon\mathscr{A}\to\mathbf{Cat}$ is a lax conical
colimit of representables; 
\item the following conditions hold: 
\begin{enumerate}
\item for every object $\left(B,y\right)$ in $\textnormal{el }F$ there
exists a lax el-generic object $\left(A,x\right)$ and morphism $\left(f,\alpha\right)\colon\left(A,x\right)\nrightarrow\left(B,y\right)$
with $\alpha$ invertible; 
\item el-generic morphisms between lax el-generic objects compose to el-generic
morphisms. 
\end{enumerate}
\end{enumerate}
\end{thm}

\begin{proof}
The direction $\left(2\right)\Rightarrow\left(1\right)$ is clear
from Proposition \ref{ff} as condition (a) means that for any $B\in\mathscr{A}$
and $y\in FB$ we have a lax el-generic $\left(A,x\right)$ and morphism
$\left(f,\alpha\right)\colon\left(A,x\right)\nrightarrow\left(B,y\right)$
in $\textnormal{el }F$ with $\alpha$ invertible, so that 
\[
\Lambda_{B}\left(A,x,f\colon A\to B\right)=Ff\left(x\right)\stackrel{\alpha}{\to}y
\]
which witnesses the essential surjectivity of $\Lambda_{B}$ at $y\in FB$.

For $\left(1\right)\Rightarrow\left(2\right)$, suppose we are given
a category $\mathfrak{M}$ and pseudofunctor $P_{\left(-\right)}\colon\mathfrak{M}\to\mathscr{A}$
(assuming without loss of generality that $P_{\left(-\right)}$ strictly
preserves identities) such that $F\simeq\int_{\textnormal{lax}}^{m\in\mathfrak{M}}\mathscr{A}\left(P_{m},-\right)$,
and consequently 
\[
\textnormal{el }F\simeq\int_{\textnormal{lax}}^{m\in\mathfrak{M}^{F}}\textnormal{el }\mathscr{A}\left(P_{m},-\right).
\]
This shows that $\textnormal{el }F$ is equivalent to the bicategory
with: 
\begin{description}
\item [{Objects}] An object is a triple of the form $\left(m\in\mathfrak{M},A\in\mathscr{A},x\colon P_{m}\to A\right)$; 
\item [{Morphisms}] The morphisms $\left(m,A,x\right)\nrightarrow\left(n,B,y\right)$
are triples comprising a morphism $u\colon m\to n$ in $\mathfrak{M}^{F}$,
a morphism $f\colon A\to B$ in $\mathscr{A}$ and a 2-cell 
\[
\xymatrix{P_{m}\myar{x}{r}\ar[d]_{P_{u}}\dtwocell[0.5]{rd}{\theta} & A\ar[d]^{f}\\
P_{n}\myard{y}{r} & B
}
\]
in $\mathscr{A}$; 
\item [{2-cells}] A 2-cell $\lambda\colon\left(u,f,\theta\right)\Rightarrow\left(u,g,\phi\right)\colon\left(m,A,x\right)\nrightarrow\left(n,B,y\right)$
is a 2-cell $\lambda\colon f\Rightarrow g$ in $\mathscr{A}$ such
that 
\[
\xymatrix{P_{m}\myar{x}{r}\ar[d]_{P_{u}}\dtwocell[0.5]{rd}{\theta} & A\ar[d]^{f} & \ar@{}[d]|-{=} & P_{m}\myar{x}{r}\ar[d]_{P_{u}}\dltwocell[0.4]{rd}{\phi} & A\ar@/_{0.6pc}/[d]_{g}\ar@/^{0.6pc}/[d]^{f}\ltwocell[0.5]{d}{\lambda}\\
P_{n}\myard{y}{r} & B & \; & P_{n}\myard{y}{r} & B\mathrlap{\ .}
}
\]
\end{description}
\noun{Existence of lax el-generics.} We first show that each 
\[
\left(m\in\mathfrak{M}^{F},P_{m}\in\mathscr{A},\textnormal{id}\colon P_{m}\to P_{m}\right)
\]
in $\textnormal{el }F$ is lax el-generic. Consider a diagram 
\[
\xymatrix{\; & \left(n,C,z\right)\myar{\left(\textnormal{id},g,\textnormal{id}\right)}{d}\\
\left(m,P_{m},\textnormal{id}\right)\myard{\left(u,f,\alpha\right)}{r}\ar@{..>}[ur]^{\left(u,h,\gamma\right)} & \left(n,B,y\right)\utwocell[0.35]{lu}{\nu}
}
\]
where $\left(u,f,\alpha\right)$ and $\left(\textnormal{id},g,\textnormal{id}\right)$
are respectively 
\[
\xymatrix{P_{m}\myar{\textnormal{id}}{r}\ar[d]_{P_{u}}\dtwocell[0.5]{rd}{\alpha} & P_{m}\ar[d]^{f} &  &  & P_{n}\myar{z}{r}\ar[d]_{P_{\textnormal{id}}}\dtwocell[0.5]{rd}{\textnormal{id}} & C\ar[d]^{g}\\
P_{n}\myard{y}{r} & B &  &  & P_{n}\myard{y}{r} & B
}
\]
then we recover a canonical $\left(u,h,\gamma\right)$ as 
\begin{equation}
\xymatrix{P_{m}\myar{\textnormal{id}}{r}\ar[d]_{P_{u}}\dtwocell[0.5]{rd}{\textnormal{id}} & P_{m}\ar[d]^{z\cdot P_{u}}\\
P_{n}\myard{z}{r} & C
}
\label{liftingconstruct}
\end{equation}
with the 2-cell $\nu\colon f\Rightarrow gh=gzP_{u}=yP_{u}$ given
as $\alpha$. Now, for universality, suppose we have a $\left(u,k,\phi\right)$
given as 
\[
\xymatrix{P_{m}\myar{\textnormal{id}}{r}\ar[d]_{P_{u}}\dtwocell[0.5]{rd}{\phi} & P_{m}\ar[d]^{k}\\
P_{n}\myard{z}{r} & C
}
\]
with a 2-cell $\psi\colon f\Rightarrow gk$ such that 
\begin{equation}
\xymatrix{P_{m}\myar{\textnormal{id}}{r}\ar[dd]_{P_{u}}\dtwocell[0.5]{rdd}{\alpha} & P_{m}\ar[dd]^{f} & \ar@{}[dd]|-{=} & P_{m}\myar{\textnormal{id}}{r}\ar[d]_{P_{u}}\dtwocell[0.5]{rd}{\phi} & P_{m}\ar[d]^{k}\ar@/^{2pc}/[dd]^{f}\\
 &  &  & P_{n}\myar{z}{r}\ar[d]_{P_{\textnormal{id}}}\dtwocell[0.5]{rd}{\textnormal{id}} & C\ar[d]^{g}\ltwocell[0.3]{r}{\psi} & \;\\
P_{n}\myard{y}{r} & B & \; & P_{n}\myard{y}{r} & B
}
\label{label}
\end{equation}
Then we can take our induced map $\lambda\colon k\Rightarrow h$ as
$\phi\colon k\Rightarrow z\cdot P_{u}$. It is trivial that 
\begin{equation}
\xymatrix{P_{m}\myar{\textnormal{id}}{r}\ar[d]_{P_{u}}\dtwocell[0.5]{rd}{\phi} & P_{m}\ar[d]^{k} & \ar@{}[d]|-{=} & P_{m}\myar{\textnormal{id}}{r}\ar[d]_{P_{u}}\dtwocell[0.5]{rd}{\textnormal{id}} & P_{m}\ar[d]|-{z\cdot P_{u}}\ar@/^{3pc}/[d]^{k} & \;\\
P_{n}\myard{z}{r} & C & \; & P_{n}\myard{z}{r} & C\ltwocell[0.5]{ru}{\lambda}
}
\label{subtermeq}
\end{equation}
so that $\lambda$ is a 2-cell $\left(u,k,\phi\right)\Rightarrow\left(u,h,\gamma\right)$.
Also, from (\ref{subtermeq}) it is clear that $\lambda=\phi$ is
the \emph{only} 2-cell $\left(u,k,\phi\right)\Rightarrow\left(u,h,\gamma\right)$,
meaning $\left(u,h,\gamma\right)$ is subterminal within its hom-category.
Moreover, (\ref{label}) shows $\psi$ pasted with $\lambda=\phi$
is $\alpha=\nu$.

\noun{Classification of lax el-generics.} We now show that an object
\[
\left(m\in\mathfrak{M}^{F},A\in\mathscr{A},x\colon P_{m}\to A\right)
\]
in $\textnormal{el }F$ is lax el-generic if and only if $x$ is an
equivalence. It is clear the above argument generalizes if one replaces
$\left(m,P_{m},\textnormal{id}\right)$ with $\left(m,A,x\right)$
where $x$ is an equivalence. Conversely, if $\left(m,A,x\right)$
is a lax el-generic object then we may construct the universal diagram
\[
\xymatrix{\; & \left(m,P_{m},\textnormal{id}\right)\myar{\left(1,x,\textnormal{id}\right)}{d}\\
\left(m,A,x\right)\myard{\left(1,1,\textnormal{id}\right)}{r}\ar@{..>}[ur]^{\left(1,x^{*},\gamma\right)} & \left(m,A,x\right)\utwocell[0.35]{lu}{\nu}
}
\]
noting that $\nu$ and $\gamma$ are both invertible. In fact, this
gives an adjoint equivalence. That $\nu$ is a 2-cell says 
\[
\xymatrix{P_{m}\myar{x}{r}\ar[dd]_{\textnormal{id}}\dtwocell[0.5]{rdd}{\textnormal{id}} & A\ar[dd]^{\textnormal{id}} &  & P_{m}\myar{x}{r}\ar[d]_{\textnormal{id}}\dltwocell[0.4]{rd}{\gamma} & A\ar[d]_{x^{*}}\ar@/^{2pc}/[dd]^{\textnormal{id}}\\
 &  & = & P_{m}\ar[d]_{\textnormal{id}}\ar[r]_{\textnormal{id}}\dltwocell[0.4]{rd}{\textnormal{id}} & P_{m}\ar[d]_{x}\ltwocell[0.37]{r}{\nu} & \;\\
P_{m}\myard{x}{r} & A &  & P_{m}\myard{x}{r} & A
}
\]
which gives one triangle identity. For the other identity, note that
2-cells $\xi\colon\left(1,x^{*}xx^{*},\gamma\gamma\right)\Rightarrow\left(1,x^{*},\gamma\right)$,
meaning 2-cells $\xi$ such that

\begin{equation}
\xymatrix{P_{m}\myar{x}{r}\ar[d]_{\textnormal{id}}\dltwocell[0.4]{rd}{\gamma\gamma} & A\ar[d]^{x^{*}xx^{*}} & \; & P_{m}\myar{x}{r}\ar[d]_{\textnormal{id}}\dltwocell[0.4]{rd}{\gamma} & A\ar[d]_{x^{*}}\ar@/^{2pc}/[d]^{x^{*}xx^{*}}\ltwocell[0.37]{dr}{\xi} & \;\\
P_{m}\myard{\textnormal{id}}{r} & P_{m} & \ar@{}[u]|-{=} & P_{m}\myard{\textnormal{id}}{r} & A & \;
}
\label{xicoh}
\end{equation}
are unique, as $\left(1,x^{*},\gamma\right)$ is subterminal within
its hom-category. But we may take $\xi$ to be 
\[
\gamma x^{*}\colon\left(1,x^{*}xx^{*},\gamma\gamma\right)\Rightarrow\left(1,x^{*},\gamma\right)
\]
or 
\[
x^{*}\nu^{-1}\colon\left(1,x^{*}xx^{*},\gamma\gamma\right)\Rightarrow\left(1,x^{*},\gamma\right)
\]
which both satisfy (\ref{xicoh}). Thus $\gamma x^{*}=x^{*}\nu^{-1}$
and so $\gamma x^{*}\cdot x^{*}\nu=\textnormal{id}$ giving the other
triangle identity.

\noun{Existence of lax el-generic factorizations.} Suppose we are
given a $\left(n,B,y\colon P_{n}\to B\right)$ in $\textnormal{el }F$.
We have the map $\left(n,P_{n},\textnormal{id}\colon P_{n}\to P_{n}\right)\nrightarrow\left(n,B,y\colon P_{n}\to B\right)$
given as 
\[
\xymatrix{P_{n}\myar{\textnormal{id}}{r}\ar[d]_{P_{\textnormal{id}}}\dtwocell[0.5]{rd}{\textnormal{id}} & P_{n}\ar[d]^{y}\\
P_{n}\myard{y}{r} & B
}
\]
which is of the required form since the 2-cell involved is invertible.

\noun{El-generic morphisms form a category. }Before showing that el-generic
morphisms form a category, we will need a characterization of them.
Now, specializing the earlier argument of ``existence of expected
lax el-generics'' to the case when $g$ is the identity (though generalizing
the identity on $P_{m}$ to an equivalence $x\colon P_{m}\to A$)
we see that if $\left(m,A,x\right)$ is el-generic (i.e. $x$ is an
equivalence) 
\[
\xymatrix{\; & \left(n,C,z\right)\myar{\left(\textnormal{id},\textnormal{id},\textnormal{id}\right)}{d}\\
\left(m,A,x\right)\myard{\left(u,f,\alpha\right)}{r}\ar@{..>}[ur]^{\left(u,h,\gamma\right)} & \left(n,B,y\right)\utwocell[0.35]{lu}{\nu}
}
\]
the lifting $\left(u,h,\gamma\right)$ above, constructed as in (\ref{liftingconstruct}),
has $\gamma$ invertible. It is also clear that if $\left(u,f,\alpha\right)$
is such that $\alpha$ is invertible, then the lifting $\left(u,h,\gamma\right)$
through $\left(\textnormal{id},\textnormal{id},\textnormal{id}\right)$
constructed as in (\ref{liftingconstruct}) is given by $\left(u,f,\alpha\right)$.

This shows that the el-generic morphisms between lax el-generic objects
are diagrams of the form 
\[
\xymatrix{P_{m}\myar{x}{r}\ar[d]_{P_{u}}\dtwocell[0.5]{rd}{\alpha} & A\ar[d]^{f}\\
P_{n}\myard{y}{r} & B
}
\]
with $\alpha$ invertible, and it is clear that these are closed under
composition and that identities are such diagrams. 
\end{proof}
\begin{rem}
When $F\colon\mathscr{A}\to\mathbf{Cat}$ is a lax conical colimit
of representables, and from a lax el-generic object $\left(A,x\right)$
we construct the universal diagram 
\[
\xymatrix{\; & \left(C,z\right)\myar{\left(g,\beta\right)}{d}\\
\left(A,x\right)\myard{\left(f,\alpha\right)}{r}\ar@{..>}[ur]^{\left(h,\gamma\right)} & \left(B,y\right)\utwocell[0.35]{lu}{\nu}
}
\]
the 2-cell $\nu$ is the \emph{unique} 2-cell $\left(f,\alpha\right)\Rightarrow\left(g,\beta\right)\cdot\left(h,\gamma\right)$.
This is since for such an $F$, el-generic morphisms compose and any
map $\left(g,\beta\right)$ with $\beta$ invertible is generic. Subterminality
of $\left(g,\beta\right)\cdot\left(h,\gamma\right)$ then gives uniqueness. 
\end{rem}

\begin{rem}
\label{Munique} When $F\colon\mathscr{A}\to\mathbf{Cat}$ is a lax
conical colimit of representables, written $F\simeq\int_{\textnormal{lax}}^{m\in\mathfrak{M}^{F}}\mathscr{A}\left(P_{m},-\right)$,
then $\mathfrak{M}^{F}$ is equivalent to the category of strict\footnote{Strict here means if both $\alpha$ and $\beta$ are identities, then
both $\nu$ and $\gamma$ are identities.} lax el-generic objects $\left(A,x\right)$ and representative el-generic
morphisms in $\textnormal{el }F$. This is a consequence of the characterization
of lax el-generic objects and morphisms given in the above proof of
Theorem \ref{p5:laxequiv}. Moreover, as Theorem \ref{p5:laxequiv}
constructs $\mathfrak{M}^{F}$ as the the category of lax el-generic
objects and morphisms, we conclude this non-strict choice of $\mathfrak{M}^{F}$
is also equivalent (to the category indexing $F$ as a lax conical
colimit of representables).
\end{rem}

The following lemma will not be used until the last section, but is
expressed in terms of el-generic objects and so we give it here.
\begin{lem}
\label{opcartequiv} Let $\mathscr{A}$ be a bicategory and let $F\colon\mathscr{A}\to\mathbf{Cat}$
be a pseudofunctor. Then every opcartesian morphism between two lax
el-generic objects $\left(h,\gamma\right)\colon\left(A,x\right)\nrightarrow\left(C,z\right)$
in $\textnormal{el }F$ is an equivalence. 
\end{lem}

\begin{proof}
Given such a $\left(h,\gamma\right)$ we may form a $\left(k,\psi\right)$
as on the left below

\[
\xymatrix{\; & \left(A,x\right)\ar[d]^{\left(h,\gamma\right)} &  & \; & \left(C,z\right)\ar[d]^{\left(k,\psi\right)}\\
\left(C,z\right)\ar[r]_{\left(1,\textnormal{id}\right)}\ar[ur]^{\left(k,\psi\right)} & \left(C,z\right)\utwocell[0.35]{ul}{\nu} &  & \left(A,x\right)\ar[r]_{\left(1,\textnormal{id}\right)}\ar[ur]^{\left(h',\gamma'\right)} & \left(A,x\right)\utwocell[0.35]{ul}{\mu}
}
\]
and one can then form a $\left(h',\gamma'\right)$ as on the right
above. As $\nu$ and $\mu$ have inverses
\[
\left(h',\gamma'\right)\cong\left(h,\gamma\right)\left(k,\psi\right)\left(h',\gamma'\right)\cong\left(h,\gamma\right)
\]
so $\left(h,\gamma\right)$ has pseudo-inverse $\left(k,\psi\right)$.
\end{proof}

\section{Lax generic factorizations and lax familial pseudofunctors\label{laxgenfacsec}}

Here we specialize the results of Section \ref{genelements} to the
case when $F\colon\mathscr{A}\to\mathbf{Cat}$ is of the form $\mathscr{B}\left(X,T-\right)$
for a pseudofunctor $T\colon\mathscr{A}\to\mathscr{B}$. The following
is a generalization of Diers' notion of familial functor (given in
Definition \ref{leftmultiadjoint}) to the case of a pseudofunctor
$T\colon\mathscr{A}\to\mathscr{B}$. 
\begin{defn}
\label{deflaxfamilial} Let $\mathscr{A}$ and $\mathscr{B}$ be bicategories
and let $T\colon\mathscr{A}\to\mathscr{B}$ be a pseudofunctor. We
say that $T$ is \emph{lax familial} if there exists a pseudofunctor
$\mathfrak{M}_{\left(-\right)}\colon\mathscr{B}^{\textnormal{op}}\to\mathbf{Cat}$
and a pseudofunctor $\mathbf{P}\colon\int_{\textnormal{lax}}^{X\in\mathscr{B}}\mathfrak{M}_{X}\to\mathscr{A}$
such that 
\[
\mathscr{B}\left(X,T-\right)\simeq\int_{\textnormal{lax}}^{m\in\mathfrak{M}_{X}}\mathscr{A}\left(P_{m}^{X},-\right)
\]
for all $X\in\mathscr{B}$, where each $P_{\left(-\right)}^{X}\colon\mathfrak{M}_{X}\to\mathscr{A}$
is obtained from $\mathbf{P}$ by composing with the inclusion $\mathfrak{M}_{X}\to\int_{\textnormal{lax}}^{X\in\mathscr{B}}\mathfrak{M}_{X}$. 
\end{defn}

\begin{rem}
One might wonder why we did not simply define $T$ to be lax familial
when every 
\[
\mathscr{B}\left(X,T-\right)\colon\mathscr{A}\to\mathbf{Cat}
\]
is a lax conical colimit of representables. The reason is that this
condition would only be sufficient to force $\mathbf{P}$ (which may
be constructed from this condition) to be a normal lax functor. 

This should not be surprising. In dimension one, simply asking that
each $\mathscr{B}\left(X,T-\right)$ is a coproduct of representables
is enough to define $T$ being familial since the indexings $\mathfrak{M}_{X}$
are sets, and thus there are no naturality conditions to consider.
However, when the indexing is a category we must account for these
naturality conditions (equivalent to ensuring that $\mathbf{P}$ is
a pseudofunctor), and so the definition of a lax familial pseudofunctor
is slightly more complicated.
\end{rem}

\begin{rem}
More abstractly, one can define familial functors as the admissible
maps against the cocompletion under coproducts \cite{WalkerYonedaKZ},
sending a category $\mathcal{A}$ to $\int_{\textnormal{lax}}^{I\in\mathbf{Set}}\mathbf{CAT}\left(I,\mathcal{A}\right)$,
and define lax familial functors as the admissible maps against the
cocompletion under lax conical colimits, sending a bicategory $\mathscr{A}$
to $\int_{\textnormal{lax}}^{\mathcal{I}\in\mathbf{Cat}}\mathbf{BICAT}\left(\mathcal{I},\mathscr{A}\right)$.
However, there are a number of technicalities here, as one should
consider opposite categories (as we are really using corepresentables),
and the theory of ``higher'' versions of KZ pseudomonads \cite{kock1995}
is not fully developed.
\end{rem}

Before applying Theorem \ref{p5:laxequiv} to $\mathbf{Cat}$-valued
presheaves of the form $\mathscr{B}\left(X,T-\right)$, we will need
the appropriate notions of genericity with respect to a pseudofunctor
$T\colon\mathscr{A}\to\mathscr{B}$. The following definitions are
recovered by specializing the definitions of genericity in the last
section, namely Definitions \ref{deflaxgenericobject} and \ref{defgenericmorphism},
to the case when $F\colon\mathscr{A}\to\mathbf{Cat}$ is of the form
$\mathscr{B}\left(X,T-\right)$ for a pseudofunctor $T\colon\mathscr{A}\to\mathscr{B}$.
\begin{defn}
\label{deflaxgenericT} Let $\mathscr{A}$ and $\mathscr{B}$ be bicategories
and let $T\colon\mathscr{A}\to\mathscr{B}$ be a pseudofunctor. Then
a 1-cell $\delta\colon X\to TA$ is \emph{lax-generic} if for any
diagram and 2-cell $\alpha$ as on the left below 
\[
\xymatrix{X\myar{z}{r}\ar[d]_{\delta} & TB\ar[d]^{Tg} & \ar@{}[d]|-{=} & X\myar{z}{r}\ar[d]_{\delta}\utwocell[0.3]{rd}{\gamma} & TB\ar[d]^{Tg}\\
TA\myard{Tf}{r}\utwocell[0.5]{ru}{\alpha} & TC & \; & TA\myard{Tf}{r}\ar[ur]|-{Th} & TC\utwocell[0.35]{ul}{T\nu}
}
\]
there exists a diagram and 2-cells $\nu$ and $\gamma$ as on the
right above\footnote{We are suppressing the pseudofunctorality constraint $Tg\cdot Th\cong Tgh$.}
which is equal to $\alpha$, such that: 
\begin{enumerate}
\item given any 2-cells $\omega,\tau\colon k\Rightarrow h$ such that
\[
\xymatrix{X\myar{z}{r}\ar[d]_{\delta}\utwocell[0.3]{rd}{\gamma}\ultwocell[0.65]{rd}{T\omega} & TB & \ar@{}[d]|-{=} & X\myar{z}{r}\ar[d]_{\delta}\utwocell[0.3]{rd}{\gamma}\ultwocell[0.65]{rd}{T\tau} & TB\\
TA\ar[ur]|-{Th}\ar@/_{1.4pc}/[ur]_{Tk} & \; & \; & TA\ar[ur]|-{Th}\ar@/_{1.4pc}/[ur]_{Tk} & \;
}
\]
we have $\omega=\tau$; 
\item given any other diagram 
\[
\xymatrix{X\myar{z}{r}\ar[d]_{\delta}\utwocell[0.3]{rd}{\phi} & TB\ar[d]^{Tg}\\
TA\myard{Tf}{r}\ar[ur]|-{Tk} & TC\utwocell[0.35]{ul}{T\psi}
}
\]
equal to $\alpha$, there exists a (necessarily unique) 2-cell $\overline{\psi}\colon k\Rightarrow h$
such that 
\[
\xymatrix{X\myar{z}{r}\ar[d]_{\delta}\utwocell[0.3]{rd}{\phi} & TB & \ar@{}[d]|-{=} & X\myar{z}{r}\ar[d]_{\delta}\utwocell[0.3]{rd}{\gamma}\ultwocell[0.65]{rd}{T\overline{\psi}} & TB\\
TA\ar[ur]|-{Tk} & \; & \; & TA\ar[ur]|-{Th}\ar@/_{1.4pc}/[ur]_{Tk} & \;
}
\]
and 
\[
\xymatrix{ & B\ar[d]^{g} & \ar@{}[d]|-{=} &  & B\ar[d]^{g}\\
A\myard{f}{r}\ar[ur]|-{k}\ar@/^{1.8pc}/[ur]^{h} & C\utwocell[0.35]{ul}{\psi}\utwocell[0.75]{ul}{\overline{\psi}} & \; & A\myard{f}{r}\ar[ur]|-{h} & C\utwocell[0.35]{ul}{\nu}\mathrlap{\ ;}
}
\]
\item if $\alpha$ is invertible, then both $\gamma$ and $\nu$ are invertible. 
\end{enumerate}
We call a factorization

\[
\xymatrix{X\myar{z}{r}\ar[d]_{\delta} & TB\ar[d]^{Tg} & \ar@{}[d]|-{=} & X\myar{z}{r}\ar[d]_{\delta}\utwocell[0.3]{rd}{\gamma} & TB\ar[d]^{Tg}\\
TA\myard{Tf}{r}\utwocell[0.5]{ru}{\alpha} & TC & \; & TA\myard{Tf}{r}\ar[ur]|-{Th} & TC\utwocell[0.35]{ul}{T\nu}
}
\]
the \emph{universal factorization} of $\alpha$ if both (1) and (2)
are satisfied above. 
\end{defn}

Earlier in Definition \ref{defgenericmorphism} we defined a 1-cell
to be generic when it satisfied a certain strong mixed lifting property.
Translating this definition into the context of a pseudofunctor $T\colon\mathscr{A}\to\mathscr{B}$
results in the below definition. 
\begin{defn}
\label{deflaxgenericcell} Let $\mathscr{A}$ and $\mathscr{B}$ be
bicategories and let $T\colon\mathscr{A}\to\mathscr{B}$ be a pseudofunctor.
Let $\delta\colon X\to TA$ be a generic 1-cell. Then a pair $\left(h,\gamma\right)$
of the form 
\[
\xymatrix@R=1em{ & TA\ar[dd]^{Th}\\
X\ar[ur]^{\delta}\ar[rd]_{z}\dltwocell[0.5]{r}{\gamma} & \;\\
 & TB
}
\]
is a \emph{generic cell} if: 
\begin{enumerate}
\item given any 2-cells $\omega,\tau\colon k\Rightarrow h$ such that 
\[
\xymatrix@R=1em{ & TA\ar[dd]|-{Th}\ar@/^{2pc}/[dd]^{Tk} &  &  & TA\ar[dd]|-{Th}\ar@/^{2pc}/[dd]^{Tk}\\
X\ar[ur]^{\delta}\ar[rd]_{z}\dltwocell[0.5]{r}{\gamma} & \;\ltwocell[0.3]{r}{T\omega} & \ar@{}[]|-{\qquad=} & X\ar[ur]^{\delta}\ar[rd]_{z}\dltwocell[0.5]{r}{\gamma} & \;\ltwocell[0.3]{r}{T\tau} & \;\\
 & TB &  &  & TB
}
\]
we have $\omega=\tau$; 
\item given any other diagram 
\[
\xymatrix@R=1em{ & TA\ar[dd]^{Tk}\\
X\ar[ur]^{\delta}\ar[rd]_{z}\dltwocell[0.5]{r}{\phi} & \;\\
 & TB
}
\]
and $\lambda\colon h\Rightarrow k$ such that 
\[
\xymatrix@R=1em{ & TA\ar[dd]^{Th} &  &  & TA\ar[dd]|-{Tk}\ar@/^{2pc}/[dd]^{Th}\\
X\ar[ur]^{\delta}\ar[rd]_{z}\dltwocell[0.5]{r}{\gamma} & \; & \ar@{}[]|-{=} & X\ar[ur]^{\delta}\ar[rd]_{z}\dltwocell[0.5]{r}{\phi} & \;\ltwocell[0.3]{r}{T\lambda} & \;\\
 & TB &  &  & TB
}
\]
there exists a (necessarily unique) $\lambda^{*}\colon k\Rightarrow h$
such that 
\[
\xymatrix@R=1em{ & TA\ar[dd]^{Tk} &  &  & TA\ar[dd]|-{Th}\ar@/^{2pc}/[dd]^{Tk}\\
X\ar[ur]^{\delta}\ar[rd]_{z}\dltwocell[0.5]{r}{\phi} & \; & \ar@{}[]|-{=} & X\ar[ur]^{\delta}\ar[rd]_{z}\dltwocell[0.5]{r}{\gamma} & \;\ltwocell[0.3]{r}{T\lambda^{*}} & \;\\
 & TB &  &  & TB
}
\]
and $\lambda^{*}\lambda=\textnormal{id}_{h}$. 
\end{enumerate}
From this definition, the following is clear. 
\end{defn}

\begin{cor}
For any universal factorization 
\[
\xymatrix{X\myar{z}{r}\ar[d]_{\delta} & TB\ar[d]^{Tg} & \ar@{}[d]|-{=} & X\myar{z}{r}\ar[d]_{\delta}\utwocell[0.3]{rd}{\gamma} & TB\ar[d]^{Tg}\\
TA\myard{Tf}{r}\utwocell[0.5]{ru}{\alpha} & TC & \; & TA\myard{Tf}{r}\ar[ur]|-{Th} & TC\utwocell[0.35]{ul}{T\nu}
}
\]
it follows that $\left(h,\gamma\right)$ is a generic 2-cell. 
\end{cor}

Before proving the main theorem of this section, it is worth defining
the spectrum of a pseudofunctor. This is to be the two-dimensional
analogue of Diers' definition of spectrum of a functor \cite[Definition 3]{DiersDiag}.

It turns out that for a lax familial functor, the reindexing $\mathbf{P}$
necessarily has domain given by the Grothendieck construction of the
spectrum, hence why the spectrum appears in this section and in the
proof of Theorem \ref{laxequivT}. 
\begin{defn}
Let $\mathscr{A}$ and $\mathscr{B}$ be bicategories and let $T\colon\mathscr{A}\to\mathscr{B}$
be a pseudofunctor such that $\mathscr{B}\left(X,T-\right)$ is a
lax conical colimit of representables for every $X\in\mathscr{B}$.\footnote{It makes sense to define the spectrum with just this assumption. However,
in most cases of interest $T$ will satisfy the stronger condition
of being lax familial.} For each $X\in\mathscr{B}$, define $\mathfrak{M}_{X}$ as the category
with objects given by lax-generic morphisms out of $X$ and morphisms
given by representative generic cells between them. We define the
\emph{spectrum} of $T$ to be the pseudofunctor 
\[
\mathbf{Spec}_{T}\colon\mathscr{B}^{\textnormal{op}}\to\mathbf{Cat}
\]
sending an object $X\in\mathscr{B}$ to $\mathfrak{M}_{X}$ and a
morphism $f\colon Y\to X$ in $\mathscr{B}$ to the functor $\mathfrak{M}_{f}\colon\mathfrak{M}_{X}\to\mathfrak{M}_{Y}$
which takes a generic morphism $\delta\colon X\to TA$ to $\delta'\colon Y\to TP$
where $\delta\cdot f\cong Tu\cdot\delta'$ is a chosen generic factorization
of $\delta\cdot f$, and takes a generic 2-cell $\gamma\colon Th\cdot\delta\Rightarrow\sigma$
as on the left below to the generic 2-cell $\overline{\gamma}\colon T\overline{h}\cdot\delta'\Rightarrow\sigma'$
as on the right below
\[
\xymatrix@R=1em{ & TP\ar[r]^{Tu}\ar@{}[d]|-{\cong} & TA\ar[dd]^{Th} &  &  & TP\ar[r]^{Tu}\ar[dd]^{T\overline{h}} & TA\ar[dd]^{Th}\\
Y\ar[r]^{f}\ar[ur]^{\delta'}\ar[rd]_{\sigma'} & X\ar[ur]^{\delta}\ar[rd]_{\sigma}\dltwocell[0.5]{r}{\gamma}\ar@{}[d]|-{\cong} & \; & = & Y\ar[ur]^{\delta'}\ar[rd]_{\sigma'}\dltwocell[0.5]{r}{\overline{\gamma}} &  & \;\\
 & TQ\ar[r]_{Tv} & TB &  &  & TQ\ar[r]_{Tv}\dltwocell[0.5]{ruu}{T\nu} & TB
}
\]
constructed as the universal factorization of the left pasting above. 
\end{defn}

\begin{rem}
When $\mathscr{A}$ has a terminal object the spectrum has an especially
simple form, namely as the functor $\mathscr{B}\left(-,T1\right)\colon\mathscr{B}^{\textnormal{op}}\to\mathbf{Cat}$. 
\end{rem}

Later on we will need to use the the Grothendieck construction of
the spectrum, which has the following relatively simple description.
\begin{lem}
\label{constructspectrum} Let $\mathscr{A}$ and $\mathscr{B}$ be
bicategories and let $T\colon\mathscr{A}\to\mathscr{B}$ be a pseudofunctor
such that $\mathscr{B}\left(X,T-\right)$ is a lax conical colimit
of representables for every $X\in\mathscr{B}$. Then the bicategory
of elements of the spectrum $\mathbf{Spec}_{T}\colon\mathscr{B}^{\textnormal{op}}\to\mathbf{Cat}$
is the bicategory 
\[
\textnormal{el }\mathfrak{M}_{\left(-\right)}\cong\int_{\textnormal{lax}}^{X\in\mathscr{B}}\mathfrak{M}_{X}
\]
consisting of: 
\begin{description}
\item [{Objects}] An object is a pair of the form $\left(X\in\mathscr{B},\delta\colon X\to TA\right)$
where $\delta$ is a generic out of $X$; 
\item [{Morphisms}] A morphism $\left(X\in\mathscr{B},\delta\colon X\to TA\right)\nrightarrow\left(Y\in\mathscr{B},\sigma\colon Y\to TB\right)$
is a morphism $f\colon X\to Y$ in $\mathscr{B}$ and a representative
generic cell $\left(h,\gamma\right)$ as below 
\[
\xymatrix{X\myar{\delta}{r}\ar[d]_{f}\dltwocell[0.5]{rd}{\gamma} & TA\ar[d]^{Th}\\
Y\myard{\sigma}{r} & TB
}
\]
\item [{2-cells}] A 2-cell $\left(f,h,\gamma\right)\Rightarrow\left(g,k,\phi\right)\colon\left(X,\delta\right)\nrightarrow\left(Y,\sigma\right)$
is a 2-cell $\nu\colon f\Rightarrow g$ in $\mathscr{B}$ such that
\[
\xymatrix{X\myar{\delta}{r}\ar@/^{0.7pc}/[d]^{f}\ar@/_{0.7pc}/[d]_{g} & TA\ar[d]^{Th}\dltwocell[0.4]{ld}{\gamma} & \ar@{}[d]|-{=} & X\myar{\delta}{r}\ar@/_{0pc}/[d]_{g}\dltwocell[0.25]{rd}{\phi} & TA\ar@/^{0.7pc}/[d]^{Th}\ar@/_{0.7pc}/[d]_{Tk}\\
Y\myard{\sigma}{r}\ltwocell[0.5]{u}{\nu} & TB & \; & Y\myard{\sigma}{r} & TB\ltwocell[0.5]{u}{T\overline{\nu}}
}
\]
for some (necessarily unique) $\overline{\nu}\colon h\Rightarrow k$. 
\end{description}
Moreover, the cartesian morphisms are precisely those $\left(f,h,\gamma\right)$
such that $\gamma$ is invertible.
\end{lem}

\begin{proof}
We know, using the formula of Definition \ref{groth} (adjusted to
the contravariant case), that $\int_{\textnormal{lax}}^{X\in\mathscr{B}}\mathfrak{M}_{\left(-\right)}$
is the bicategory with objects pairs $\left(X\in\mathscr{B},m\in\mathfrak{M}_{X}\right)$,
morphisms $\left(X\in\mathscr{B},m\in\mathfrak{M}_{X}\right)\nrightarrow\left(Y\in\mathscr{B},n\in\mathfrak{M}_{Y}\right)$
given by a 1-cell $f\colon X\to Y$ and morphism $\alpha\colon m\to Ff\left(n\right)$
in $\mathfrak{M}_{X}$, and 2-cells $\nu\colon\left(f,\alpha\right)\Rightarrow\left(g,\beta\right)$
those 2-cells $\nu\colon f\Rightarrow g$ such that 
\[
\xymatrix@R=1em{m\ar@/_{1.2pc}/[rr]_{\beta}\myar{\alpha}{r} & Ff\left(n\right)\myar{\left(F\nu\right)_{n}}{r} & Fg\left(n\right)}
\]
commutes. The objects are clearly as desired. Thus a morphism $\left(X\in\mathscr{B},\delta\colon X\to TA\right)\nrightarrow\left(Y\in\mathscr{B},\sigma\colon Y\to TB\right)$
consists of an $f\colon X\to Y$ and an $\alpha\colon\delta\to\mathfrak{M}_{f}\left(\sigma\right)$
in $\mathfrak{M}_{X}$. Hence a morphism is a pair $\left(f,\left(s,\xi\right)\right)$
as below 
\[
\xymatrix@=1em{X\myar{\delta}{rr}\ar[dd]_{f}\myard{\sigma_{f}}{rd} & \;\dltwocell[0.35]{d}{\xi} & TA\ar[dl]^{Ts}\\
 & TH\myar{T\overline{f}}{rd}\\
Y\myard{\sigma}{rr}\ar@{}[ru]|-{\cong} &  & TB
}
\]
where $\left(s,\xi\right)$ is a representative generic cell, and
$T\overline{f}\cdot\sigma_{f}$ is the chosen generic factorization
of $\sigma\cdot f$. Using that generic cells $\left(s,\xi\right)$
remain generic when composed with opcartesian cells $\left(\overline{f},\cong\right)$
(because opcartesian cells are themselves generic), the above diagram
is itself a generic cell, isomorphic to a unique representative generic
cell 
\[
\xymatrix{X\myar{\delta}{r}\ar[d]_{f}\dltwocell[0.5]{rd}{\gamma} & TA\ar[d]^{Th}\\
Y\myard{\sigma}{r} & TB
}
\]
Conversely, one may form the representative generic factorization
of $\gamma$ 
\[
\xymatrix@=1em{X\myar{\delta}{rr}\ar[dd]_{\sigma_{f}} & \; & TA\ar[dd]^{Th}\ar[lldd]|-{Ts}\\
 & \ltwocell[0.4]{ul}{\xi}\\
TH\myard{T\overline{f}}{rr} &  & TB\ltwocell[0.65]{ul}{T\zeta}
}
\]
to recover $\left(s,\xi\right)$ (note that $\zeta$ is invertible
as genericity of $\left(s,\xi\right)$ is preserved by $\left(\overline{f},\textnormal{id}\right)$
and $\gamma$ is generic). That the assignment $\left(s,\xi\right)\mapsto\left(h,\gamma\right)$
defines a bijection is a consequence of the fact that any 2-cells
factors through a unique representative generic 2-cell (once a choice
of a generic 1-cell factoring each general 1-cell has been made).

It is also worth noting that the opcartesian morphisms, corresponding
to the case where $\left(s,\xi\right)$ is an equivalence (meaning
$s$ is an equivalence and $\xi$ is invertible), are those squares
where $\gamma$ is invertible. 

Finally, a 2-cell $\nu\colon\left(f,s,\xi\right)\Rightarrow\left(g,u,\theta\right)$
consists of a 2-cell $\nu\colon f\Rightarrow g$ such that 
\begin{equation}
\xymatrix@R=1em{\delta\ar@/_{1.2pc}/[rr]_{\left(u,\theta\right)}\myar{\left(s,\xi\right)}{r} & \sigma_{f}\myar{\left(\mathfrak{M}_{\nu}\right)_{\sigma}}{r} & \sigma_{g}}
\label{mv}
\end{equation}
commutes, where $\left(\mathfrak{M}_{\nu}\right)_{\sigma}$ is given
by the representative generic factorization, denoted by the pair $\left(m,\varphi\right)$,
as in the diagram below 
\[
\xymatrix@R=1em{ & TH\ar@/^{0.7pc}/[rd]^{T\overline{f}}\ar@{}[d]|-{\cong} &  &  &  & TT\ar@/^{0.7pc}/[rd]^{T\overline{f}}\ar[dd]|-{Tm}\\
X\ar@/^{0.7pc}/[r]^{f}\ar@/_{0.7pc}/[r]_{g}\dtwocell[0.4]{r}{\nu}\ar@/^{0.7pc}/[ur]^{\sigma_{f}}\ar@/_{0.7pc}/[rd]_{\sigma_{g}} & Y\ar[r]^{\sigma}\ar@{}[d]|-{\cong} & TB & = & X\dtwocell[0.4]{r}{\varphi}\ar@/^{0.7pc}/[ur]^{\sigma_{f}}\ar@/_{0.7pc}/[rd]_{\sigma_{g}} & \dtwocell[0.4]{r}{T\lambda} & TB\\
 & TS\ar@/_{0.7pc}/[ru]_{T\overline{g}} &  &  &  & TS\ar@/_{0.7pc}/[ru]_{T\overline{g}}
}
\]
Hence given such a $\nu$ we have 
\[
\xymatrix@=1.5em{X\myar{\delta}{rr}\ar[dd]_{\sigma_{g}} & \; & TA\ar[dd]^{Th}\ar[ldld]|-{Tu} &  & X\myar{\delta}{rr}\ar[dd]_{\sigma_{g}}\ar[rd]|-{\sigma_{f}} & \; & TA\ar[dd]^{Th}\ar[ld]|-{Ts}\\
\dltwocell[0.4]{ru}{\theta} & \dltwocell[0.4]{rd}{T\tau} & \; & = & \dltwocell[0.4]{r}{\varphi} & TH\ar[rd]|-{T\overline{f}}\ar[ld]|-{Tm}\dltwocell[0.7]{u}{\xi}\dltwocell[0.6]{d}{T\lambda}\dltwocell[0.6]{r}{T\zeta} & \;\\
TS\myard{T\overline{g}}{rr} & \; & TB &  & TS\myard{T\overline{g}}{rr} & \; & TB
}
\]
for some (necessarily unique) $\tau\colon h\Rightarrow\overline{g}\cdot u$.
Moreover, given a diagram as above we can take the representative
generic factorization to recover (\ref{mv}). 
\end{proof}
We can now apply Theorem \ref{p5:laxequiv} to the case where $F\colon\mathscr{A}\to\mathbf{Cat}$
is of the form $\mathscr{B}\left(X,T-\right)$ for a pseudofunctor
$T\colon\mathscr{A}\to\mathscr{B}$ to prove the following theorem. 
\begin{thm}
\label{laxequivT} Let $\mathscr{A}$ and $\mathscr{B}$ be bicategories
and let $T\colon\mathscr{A}\to\mathscr{B}$ be a pseudofunctor. Then
the following are equivalent: 
\begin{enumerate}
\item the pseudofunctor $T\colon\mathscr{A}\to\mathscr{B}$ is lax familial; 
\item the following conditions hold: 
\begin{enumerate}
\item for every object $X\in\mathscr{A}$ and 1-cell $y\colon X\to TC$
in $\mathscr{B}$, there exists a lax-generic morphism $\delta\colon X\to TA$
and 1-cell $f\colon A\to C$ such that $Tf\cdot\delta\cong y$; 
\item for any triple of lax-generic morphisms $\delta,$ $\sigma$ and $\omega$,
and pair of generic cells $\left(h,\theta\right)$ and $\left(k,\phi\right)$
as below 
\begin{equation}
\xymatrix{X\ar[d]_{\delta}\myar{f}{r}\urtwocell[0.5]{rd}{\theta} & Y\ar[d]|-{\sigma}\myar{g}{r}\urtwocell[0.5]{rd}{\phi} & Z\ar[d]^{\omega}\\
TA\myard{Th}{r} & TB\myard{Tk}{r} & TC
}
\label{dbsquare}
\end{equation}
the above pasting $\left(kh,\phi f\cdot\theta\right)$ is a generic
cell.\footnote{Suppressing pseudofunctoriality constraints of $T$.} 
\end{enumerate}
\end{enumerate}
\end{thm}

\begin{proof}
$\left(1\right)\Rightarrow\left(2\right)\colon$ Supposing that $T$
is lax familial, it follows that each $\mathscr{B}\left(X,T-\right)$
is a lax conical colimit of representables. By Theorem \ref{p5:laxequiv},
we have (2)(a), as well as 2(b) when $f$ and $g$ are both the identity
at $X$. To get the full version of (2)(b) we use that 
\[
\mathbf{P}\colon\int_{\textnormal{lax}}^{X\in\mathscr{B}}\mathfrak{M}_{X}\to\mathscr{A}
\]
is a pseudofunctor, where we have assumed without loss of generality
that each $\mathfrak{M}_{X}$ is the category of generic morphisms
out of $X$ and representative cells, using Remark \ref{Munique}.
Indeed, $\int_{\textnormal{lax}}^{X\in\mathscr{B}}\mathfrak{M}_{X}$
is the bicategory with objects pairs $\left(X,\delta\colon X\to TA\right)$
and morphisms $\left(X,\delta\colon X\to TA\right)\nrightarrow\left(Y,\sigma\colon Y\to TB\right)$
given by triples $\left(f,h,\theta\right)$ as below 
\[
\xymatrix{X\ar[d]_{\delta}\myar{f}{r}\urtwocell[0.5]{rd}{\theta} & Y\ar[d]|-{\sigma}\\
TA\myard{Th}{r} & TB
}
\]
such that $\left(h,\theta\right)$ is a generic cell. As the lax functoriality
constraints of $\mathbf{P}$ are given by factoring diagrams such
as (\ref{dbsquare}) though a generic, the invertibility of these
lax constraints of $\mathbf{P}$ forces (2)(b).

$\left(2\right)\Rightarrow\left(1\right)\colon$ Applying Theorem
\ref{p5:laxequiv} to the conditions 2(a) and 2(b) (only needing the
case when $f$ and $g$ are identities at $X$), it follows that we
may write 
\[
\mathscr{B}\left(X,T-\right)\simeq\int_{\textnormal{lax}}^{m\in\mathfrak{M}_{X}}\mathscr{A}\left(P_{m}^{X},-\right)
\]
where $\mathfrak{M}_{X}$ is the category of generic morphisms out
of $X$ and representative generic cells between them. From this,
we recover the spectrum $\mathbf{Spec}_{T}\colon\mathscr{B}^{\textnormal{op}}\to\mathbf{Cat}$
taking each $X$ to $\mathfrak{M}_{X}$. Also, we again have the canonical
normal lax functor 
\[
\mathbf{P}\colon\int_{\textnormal{lax}}^{X\in\mathscr{B}}\mathfrak{M}_{X}\to\mathscr{A}
\]
defined as in the previous implication $\left(1\right)\Rightarrow\left(2\right)$.
Using the full version of (2)(b), meaning we are no longer just using
the case when $f$ and $g$ are identities at $X$, forces this to
be a pseudofunctor (not just a normal lax functor) as required. 
\end{proof}
Under the conditions of this theorem, we also have a notion of generic
factorizations on 2-cells, in a sense we now describe. 
\begin{rem}
Suppose $T$ is lax familial, $\delta$ and $\sigma$ are generic
objects, and consider a 2-cell $\alpha\colon Tf\cdot\delta\Rightarrow Tg\cdot\sigma$.
Then $\alpha$ has a lax generic factorization 
\[
\xymatrix@R=1em{ & TA\ar@/^{0.4pc}/[rd]^{Tf}\ar@{}[dd]|-{\Downarrow\alpha} &  &  &  & TA\ar@/^{0.4pc}/[rd]^{Tf}\ar[dd]|-{Th}\\
X\ar@/^{0.4pc}/[ur]^{\delta}\ar@/_{0.4pc}/[dr]_{\sigma} &  & TC & = & X\ar@/^{0.4pc}/[ur]^{\delta}\ar@/_{0.4pc}/[dr]_{\sigma}\ar@{}[r]|-{\Downarrow\gamma} & \;\ar@{}[r]|-{\Downarrow T\nu} & TC\\
 & TB\ar@/_{0.4pc}/[ru]_{Tg} &  &  &  & TB\ar@/_{0.4pc}/[ru]_{Tg}
}
\]
Also note that any map $k\colon X\to TC$ can be factored as $T\overline{k}\cdot\xi$
for some generic $\xi$ and morphism $\overline{k}$, and so when
$T$ is surjective on objects we have a generic factorization of every
1-cell and 2-cell in the bicategory $\mathscr{B}$.
\end{rem}

\section{Comparing to Weber's familial 2-functors\label{webersection}}

The purpose of this section is to compare our definition of a lax
familial 2-functor $T\colon\mathscr{A}\to\mathscr{B}$ between 2-categories
(meaning Definition \ref{deflaxfamilial} specialized to 2-categories
and 2-functors), with Weber's definition of familial 2-functor (which
requires that $\mathscr{A}$ has a terminal object). It turns out
that these two definitions are essentially equivalent. Note also that
Weber's definition assumes some ``strictness conditions'' (such
as identity 2-cells factoring into identity 2-cells) which are natural
conditions on 2-functors, but arguably less natural in the case of
pseudofunctors.

In one dimension, a functor $T\colon\mathcal{A}\to\mathcal{B}$ (where
$\mathcal{A}$ has a terminal object) is said to be a parametric right
adjoint (or a local right adjoint) when the canonical functor $T_{1}\colon\mathcal{A}\cong\mathcal{A}/\mathbf{1}\rightarrow\mathcal{B}/T\mathbf{1}$
is a right adjoint \cite{StreetPetit}. The following is what Weber
refers to as the ``naive'' 2-categorical analogue of parametric
right adjoint \cite{Webfam}.
\begin{defn}
Suppose $\mathscr{A}$ and $\mathscr{B}$ are 2-categories, and that
$\mathscr{A}$ has a terminal object. We say a 2-functor $T\colon\mathscr{A}\to\mathscr{B}$
is a \emph{naive parametric right adjoint} if every canonical functor
(on the 2-slices) $T_{1}\colon\mathscr{A}\cong\mathscr{A}/\mathbf{1}\rightarrow\mathscr{B}/T\mathbf{1}$
is a right 2-adjoint.
\end{defn}

We now recall the notion of generic morphism corresponding to this
``naive'' 2-categorical analogue of parametric right adjoints \cite{Webfam}. 
\begin{defn}
Suppose $\mathscr{A}$ and $\mathscr{B}$ are 2-categories. Given
a 2-functor $T\colon\mathscr{A}\to\mathscr{B}$ we say a morphism
$x\colon X\to TA$ is \emph{naive-generic} if: 
\begin{enumerate}
\item for any commuting square as on the left below 
\[
\xymatrix{X\myar{z}{r}\ar[d]_{x} & TB\ar[d]^{Tg} &  &  & X\myar{z}{r}\ar[d]_{x} & TB\ar[d]^{Tg}\\
TA\myard{Tf}{r} & TC & \; &  & TA\myard{Tf}{r}\ar[ur]|-{Th} & TC
}
\]
there exists a unique $h\colon A\to B$ such that $Th\cdot x=z$ and
$f=gh$; 
\item for two commuting diagrams 
\[
\xymatrix{X\myar{z_{1}}{r}\ar[d]_{x} & TB\ar[d]^{Tg} &  &  & X\myar{z_{2}}{r}\ar[d]_{x} & TB\ar[d]^{Tg}\\
TA\myard{Tf}{r}\ar[ur]|-{Th_{1}} & TC &  &  & TA\myard{Tf}{r}\ar[ur]|-{Th_{2}} & TC
}
\]
the 2-cells $\theta\colon z_{1}\Rightarrow z_{2}$ such that $Tg\cdot\theta=\textnormal{id}$
bijectively correspond to 2-cells $\overline{\theta}\colon h_{1}\Rightarrow h_{2}$
such that $T\left(\overline{\theta}\right)\cdot x=\theta$ and $g\cdot\overline{\theta}=\textnormal{id}$. 
\end{enumerate}
\end{defn}

From this one can prove the following expected result \cite{Webfam}.
\begin{prop}
\cite{Webfam} Suppose $\mathscr{A}$ and $\mathscr{B}$ are 2-categories,
and that $\mathscr{A}$ has a terminal object. Then a 2-functor $T\colon\mathscr{A}\to\mathscr{B}$
is a naive parametric right adjoint if and only if every $f\colon X\to TA$
factors as $T\overline{f}\cdot x$ for a naive-generic morphism $x$. 
\end{prop}

Weber's actual definition of familial 2-functors (which we will soon
recall) requires certain maps in a 2-category to be fibrations \cite{Webfam}.
Thus we will need to recall the definition of fibration in a 2-category
$\mathscr{B}$. Note that when $\mathscr{B}$ is finitely complete
there are other equivalent characterizations of fibrations \cite{fibbicaty}. 
\begin{defn}
\label{defstreetfib}We say a morphism $p\colon E\to B$ in a 2-category
$\mathscr{B}$ is a \emph{fibration} if: 
\begin{enumerate}
\item for every $X\in\mathscr{B}$, the functor $\mathscr{B}\left(X,p\right)\colon\mathscr{B}\left(X,E\right)\to\mathscr{B}\left(X,B\right)$
is a fibration; 
\item for every $f\colon X\to Y$ in $\mathscr{B}$, the functor $\mathscr{B}\left(f,E\right)\colon\mathscr{B}\left(Y,E\right)\to\mathscr{B}\left(X,E\right)$
preserves cartesian morphisms. 
\end{enumerate}
If we have a choice of cartesian lifts which strictly respects composition
and identities we say the fibration \emph{splits}. 
\end{defn}

We now have the required background to define familial 2-functors
in the sense of Weber.
\begin{defn}
Suppose $\mathscr{A}$ and $\mathscr{B}$ are 2-categories and that
$\mathscr{A}$ has a terminal object. We say a 2-functor $T\colon\mathscr{A}\to\mathscr{B}$
is\emph{ Weber-familial} if 
\begin{enumerate}
\item $T$ is a naive parametric right adjoint; 
\item for every $A\in\mathscr{A}$, and unique $t_{A}\colon A\to1$ in $\mathscr{A}$,
the morphism $Tt_{A}\colon TA\to T1$ is a split fibration in $\mathscr{B}$. 
\end{enumerate}
\end{defn}

The following is Weber's analogue of lax-generic morphisms. 
\begin{defn}
Suppose $\mathscr{A}$ and $\mathscr{B}$ are 2-categories, and that
$\mathscr{A}$ has a terminal object. Given a 2-functor $T\colon\mathscr{A}\to\mathscr{B}$
for which each $Tt_{A}\colon TA\to T1$ is a split fibration, we say
a morphism $x\colon X\to TA$ is \emph{Weber-lax-generic} if for any
2-cell $\alpha$ as on the left below,

\[
\xymatrix{X\myar{z}{r}\ar[d]_{x} & TB\ar[d]^{Tg} & \ar@{}[d]|-{=} & X\myar{z}{r}\ar[d]_{x}\utwocell[0.3]{rd}{\gamma} & TB\ar[d]^{Tg}\\
TA\myard{Tf}{r}\utwocell[0.5]{ru}{\alpha} & TC & \; & TA\myard{Tf}{r}\ar[ur]|-{Th} & TC\utwocell[0.35]{ul}{T\nu}
}
\]
there exists a unique factorization $\left(h,\gamma,\nu\right)$ as
above such that $\left(h,\gamma\right)$ is chosen\footnote{Recall part of the data of a split fibration is a choice of coherent
cartesian lifts.} $Tt_{B}\colon TB\to T1$ cartesian.\footnote{This definition of lax-generics has the downside that it assumes some
of the conditions for a 2-functor being familial for it to make sense
(namely that each $Tt_{A}\colon TA\to T1$ is a split fibration),
thus not allowing for a theorem describing an equivalence between
a 2-functor being familial and admitting lax-generic factorizations.}
\end{defn}

The following lemma shows that for Weber-familial 2-functors $T$,
the lax-generics of both our sense and Weber's coincide, and our generic
2-cells can equivalently be characterized as certain cartesian morphisms. 
\begin{lem}
\label{famlemma} Suppose $\mathscr{A}$ and $\mathscr{B}$ are 2-categories
and that $\mathscr{A}$ has a terminal object. Let $T\colon\mathscr{A}\to\mathscr{B}$
be a Weber-familial 2-functor. Define $\mathfrak{M}$ as the category
with objects given by chosen naive-generics $\delta\colon X\to TA$,\footnote{Here ``chosen'' means that it is to be identified with another naive-generic
$\sigma\colon X\to TB$ if there exists a pair $\left(h,\gamma\right)$
as in \eqref{hgammaref} with $h$ invertible and $\gamma$ an identity;
thus it is a choice of a representative of an equivalence class of
naive-generics.} and morphisms given by pairs $\left(h,\gamma\right)$ 
\begin{equation}
\xymatrix@R=1em{ & TA\ar[dd]^{Th}\\
X\ar[ur]^{\delta}\ar[rd]_{\sigma}\dltwocell[0.5]{r}{\gamma} & \;\\
 & TB
}
\label{hgammaref}
\end{equation}
where $\gamma$ is chosen $Tt_{B}\colon TB\to T1$ cartesian. Then: 
\begin{enumerate}
\item for every $X\in\mathscr{B}$ we have isomorphisms 
\[
\mathscr{B}\left(X,T-\right)\cong\int_{\textnormal{lax}}^{m\in\mathfrak{M}}\mathscr{A}\left(P_{m},-\right);
\]
\item a map $\delta\colon X\to TA$ in $\mathscr{B}$ is naive-generic if
and only if it is strict\footnote{By strict we mean identity 2-cells universally factor into identity
2-cells.} lax-generic; 
\item a 2-cell in $\mathscr{B}$ as below 
\[
\xymatrix@R=1em{ & TA\ar[dd]^{Th}\\
X\ar[ur]^{\delta}\ar[rd]_{z}\dltwocell[0.5]{r}{\gamma} & \;\\
 & TB
}
\]
is generic if and only if it is $Tt_{B}\colon TB\to T1$ cartesian. 
\end{enumerate}
\end{lem}

\begin{proof}
$\left(1\right)\colon$ It suffices to check that the functors 
\[
\int_{\textnormal{lax}}^{m\in\mathfrak{M}}\mathscr{A}\left(P_{m},W\right)\to\mathscr{B}\left(X,TW\right)
\]
are isomorphisms. That this assignment is bijective on objects is
a consequence of the well-known one-dimensional case (for instance,
see \cite[Prop.~7]{WalkerGeneric}). That the assignment on morphisms
\[
\xymatrix{ & TP_{m}\ar[dd]^{Th} & P_{m}\ar[dd]_{h}\ar[rd]^{f} &  &  &  & TP_{m}\ar[dd]|-{Th}\ar[rd]^{Tf}\\
X\ar[rd]_{\delta_{m'}}\ar[ru]^{\delta_{m}} & \;\ar@{}[l]|-{\Downarrow\alpha} & \; & W\;\ar@{}[l]|-{\Downarrow\beta} & \mapsto & X\ar[rd]_{\delta_{m'}}\ar[ru]^{\delta_{m}} & \;\ar@{}[l]|-{\Downarrow\alpha} & TW\ar@{}[l]|-{\Downarrow T\beta}\\
 & TP_{m'} & P_{m'}\ar[ur]_{g} &  &  &  & TP_{m'}\ar[ur]_{Tg}
}
\]
is bijective follows from the fact each naive-generic is Weber-lax
generic \cite[Lemma 5.8]{Webfam}. Naturality is also an easy consequence
of this fact.

$\left(2\right)\colon$ If $\delta$ is naive-generic, and thus isomorphic
to a representative naive-generic, then $\delta$ is lax-generic by
(1). If $\delta$ is strict lax-generic, then from a $\theta\colon z_{1}\Rightarrow z_{2}$
we have a universal factorization 
\[
\xymatrix{X\myar{x}{r}\ar[d]_{x}\utwocell[0.5]{rd}{\theta} & TA\ar[d]^{Th_{2}} &  & \; & X\myar{x}{r}\ar[d]_{x}\utwocell[0.25]{rd}{\textnormal{id}} & TA\ar[d]^{Th_{2}}\\
TA\myard{Th_{1}}{r} & TB & \ar@{}[ru]|-{=} &  & TA\myard{Th_{1}}{r}\ar[ur]|-{T1} & TB\utwocell[0.35]{ul}{T\overline{\theta}}
}
\]
where we have used that $Tg\cdot\theta$ is an identity to see the
top right triangle above can be taken as an identity. In this way,
we recover the bijection required of a naive-generic.

$\left(3\right)\colon$ Consider a 2-cell 
\[
\xymatrix@R=1em{ & TA\ar[dd]^{Th}\\
X\ar[ur]^{\delta}\ar[rd]_{z}\dltwocell[0.5]{r}{\gamma} & \;\\
 & TB
}
\]

If this 2-cell is generic, then we have a factorization 
\begin{equation}
\xymatrix{X\myar{z}{r}\ar[d]_{\delta}\utwocell[0.5]{rd}{\gamma} & TA\ar[d]^{T\textnormal{id}} &  & \; & X\myar{z}{r}\ar[d]_{\delta}\utwocell[0.25]{rd}{\phi} & TA\ar[d]^{T\textnormal{id}}\\
TA\myard{Th}{r} & TB & \ar@{}[ru]|-{=} &  & TA\myard{Th}{r}\ar[ur]|-{Tk} & TB\utwocell[0.35]{ul}{T\lambda}
}
\label{p5:diag1}
\end{equation}
where $\phi$ is chosen cartesian. By genericity of $\gamma$, we
have an $\lambda^{*}\colon k\Rightarrow h$ such that 
\begin{equation}
\xymatrix@R=1em{ & TA\ar[dd]^{Tk} &  &  & TA\ar[dd]|-{Th}\ar@/^{2pc}/[dd]^{Tk}\\
X\ar[ur]^{\delta}\ar[rd]_{z}\dltwocell[0.5]{r}{\phi} & \; & \ar@{}[]|-{=} & X\ar[ur]^{\delta}\ar[rd]_{z}\dltwocell[0.5]{r}{\gamma} & \;\ltwocell[0.3]{r}{T\lambda^{*}} & \;\\
 & TB &  &  & TB
}
\label{p5:diag2}
\end{equation}
and $\lambda^{*}\lambda=\textnormal{id}_{h}$. Substituting (\ref{p5:diag1})
into (\ref{p5:diag2}) and using that $\delta$ is Weber-lax-generic
gives $\lambda\lambda^{*}=\textnormal{id}_{k}$. Conversely, if this
2-cell is cartesian we then have a factorization 
\[
\xymatrix{X\myar{z}{r}\ar[d]_{\delta}\utwocell[0.5]{rd}{\gamma} & TA\ar[d]^{T\textnormal{id}} &  & \; & X\myar{z}{r}\ar[d]_{\delta}\utwocell[0.25]{rd}{\phi} & TA\ar[d]^{T\textnormal{id}}\\
TA\myard{Th}{r} & TB & \ar@{}[ru]|-{=} &  & TA\myard{Th}{r}\ar[ur]|-{Tk} & TB\utwocell[0.35]{ul}{T\lambda}
}
\]
where $\left(k,\phi\right)$ is a generic 2-cell (which must also
be cartesian by the above argument). Since $\phi$ and $\gamma$ are
cartesian, and thus isomorphic to chosen cartesian morphisms, it follows
that the comparison $\lambda$ is invertible.
\end{proof}
Finally, we give the main result of this section, showing that for
2-functors $T\colon\mathscr{A}\to\mathscr{B}$ our notion of a  lax
familial 2-functor is essentially equivalent to Weber's definition. 
\begin{thm}
\label{webermain} Suppose $\mathscr{A}$ and $\mathscr{B}$ are 2-categories
and that $\mathscr{A}$ has a terminal object. Then for a 2-functor
$T\colon\mathscr{A}\to\mathscr{B}$ the following are equivalent: 
\begin{enumerate}
\item $T$ is Weber-familial; 
\item $T$ is strictly\footnote{By strict we mean isomorphic to a lax conical colimit of representables
in place of equivalent, and that the reindexings $P_{\left(-\right)}^{X}$
are 2-functors instead of pseudofunctors.} lax familial.
\end{enumerate}
\end{thm}

\begin{proof}
$\left(1\right)\Rightarrow\left(2\right)\colon$ Supposing $T\colon\mathscr{A}\to\mathscr{B}$
is Weber-familial, we have that each $\mathscr{B}\left(X,T-\right)$
is a lax conical colimit of representables by Lemma \ref{famlemma}
part (1). Also, as the generic 2-cells may be identified with the
cartesian 2-cells, we know since the fibration $Tt_{B}\colon TB\to T1$
respects precomposition (meaning it satisfies part 2 of Definition
\ref{defstreetfib}) we have the following property: for any generic
2-cell out of an $X\in\mathscr{B}$ as on the left below 
\begin{equation}
\xymatrix@R=1em{ & TA\ar[dd]^{Th} &  &  &  & TA\ar[dd]^{Th}\\
X\ar[ur]^{\delta}\ar[rd]_{z}\dltwocell[0.5]{r}{\gamma} & \; &  & Y\ar[r]^{k} & X\ar[ur]^{\delta}\ar[rd]_{z}\dltwocell[0.5]{r}{\gamma} & \;\\
 & TB &  &  &  & TB
}
\label{extraproperty}
\end{equation}
and map $k\colon Y\to X$ in $\mathscr{B}$, the right diagram is
a generic 2-cell. It is this property (along with closure of generic
cells under composition) which gives (2)(b) of Theorem \ref{laxequivT}.

$\left(2\right)\Rightarrow\left(1\right)\colon$ Suppose $T\colon\mathscr{A}\to\mathscr{B}$
is strictly lax familial. Then $T$ is a naive parametric right adjoint
since $T$ has strict lax generic factorizations, and lax-generic
implies naive generic (shown in part (2) of Lemma \ref{famlemma}).

It remains to check that each $Tt_{A}\colon TA\to T1$ is a split
fibration. To see this, note that for each $X\in\mathscr{B}$ the
functor $\mathscr{B}\left(X,TA\right)\to\mathscr{B}\left(X,T1\right)$
may be written as the functor 
\[
\int_{\textnormal{lax}}^{m\in\mathfrak{M}}\mathscr{A}\left(P_{m},A\right)\to\int_{\textnormal{lax}}^{m\in\mathfrak{M}}\mathscr{A}\left(P_{m},1\right)\cong\mathfrak{M}
\]
defined by the assignment 
\[
\xymatrix{m\ar[dd]_{\lambda} & P_{m}\ar[dd]_{P_{\lambda}}\ar[rd]^{f} &  &  & m\ar[dd]^{\lambda}\\
 & \; & A\;\ar@{}[l]|-{\Downarrow\beta} & \mapsto\\
m' & P_{m'}\ar[ur]_{g} &  &  & m'\mathrlap{\ .}
}
\]
It is straightforward to verify that for each $\left(m',g\colon P_{m}'\to A\right)$
and $\lambda\colon m\to m'$ we recover a cartesian lift as on the
left below
\[
\xymatrix{m\ar[dd]_{\lambda} & P_{m}\ar[dd]_{P_{\lambda}}\ar[rd]^{g\cdot P_{\lambda}} &  &  & m\ar[dd]^{\lambda}\\
 & \; & A\;\ar@{}[l]|-{\Downarrow\textnormal{id}} & \mapsto\\
m' & P_{m'}\ar[ur]_{g} &  &  & m'
}
\]
and it is clear the canonical choice of cartesian lifts given above
splits. The cartesian morphisms are diagrams as above (if the identity
2-cell is replaced by an isomorphism it is still cartesian), and these
correspond to generic cells in $\mathscr{B}\left(X,TA\right)$. That
for each $k\colon Y\to X$ the functor $\mathscr{B}\left(k,TA\right)\colon\mathscr{B}\left(X,TA\right)\to\mathscr{B}\left(Y,TA\right)$
preserves cartesian morphisms then follows from condition (2)(b) of
Theorem \ref{laxequivT}. 
\end{proof}

\section{Examples of familial pseudofunctors\label{examples}}

We will first consider some simple examples of lax familial pseudofunctors
which concern pseudofunctors $T\colon\mathscr{A}\to\mathscr{B}$ where
$\mathscr{A}$ is a 1-category. Our first and simplest examples of
such pseudofunctors $T\colon\mathscr{A}\to\mathscr{B}$ concern the
canonical embeddings into bicategories of spans and polynomials.

The reader will also recall that in this setting where $\mathscr{A}$
is a 1-category, $\textnormal{el }F\cong\textnormal{el }\mathscr{B}\left(X,T-\right)$
is a 1-category for each $X\in\mathscr{B}$, and so the mixed lifting
properties become the usual lifting properties. Indeed, it is clear
that Definition \ref{deflaxgenericcell} becomes trivial in this case,
so that every pair $\left(h,\gamma\right)$ out of a lax-generic $\delta$
is a generic cell.
\begin{example}
\label{spanembfam} The bicategory of spans $\mathbf{Span}\left(\mathcal{E}\right)$
in a category $\mathcal{E}$ with pullbacks was introduced by B\'enabou
\cite{ben1967}, and admits a canonical embedding $T\colon\mathcal{E}\to\mathbf{Span}\left(\mathcal{E}\right)$
sending a morphism $f$ in $\mathcal{E}$ to the span $\left(1,f\right)$.
It is interesting to note that this pseudofunctor is lax familial.
To see this, first observe that a span $X\nrightarrow TA$ is generic
if it is isomorphic to a span of the form 
\[
\xymatrix@R=1em{ & TA\ar[ld]_{s}\ar[rd]^{\textnormal{id}}\\
X &  & TA
}
\]
This is since if a span $\left(s,t\right)$ is generic we can then
factor the diagram on the left below 
\[
\xymatrix{X\myar{\left(s,1\right)}{r}\ar[d]_{\left(s,t\right)} & TM\ar[d]^{Tt} & \ar@{}[d]|-{=} & X\myar{\left(s,1\right)}{r}\ar[d]_{\left(s,t\right)}\utwocell[0.3]{rd}{\gamma} & TM\ar[d]^{Tt}\\
TA\myard{T\textnormal{id}}{r}\utwocell[0.5]{ru}{\textnormal{id}} & TA & \; & TA\myard{T\textnormal{id}}{r}\ar[ur]|-{Tu} & TA\utwocell[0.35]{ul}{T\nu}
}
\]
as on the right above, where $\nu$ is necessarily an identity (because
the domain of $T$ is a 1-category) and $\gamma$ invertible. Hence
$tu=\textnormal{id}$ and $ut$ is invertible, showing that $t$ is
invertible. Conversely, to see such a span $\left(s,1\right)$ is
generic, note that any diagram as on the left below 
\[
\xymatrix{X\myar{\left(u,v\right)}{r}\ar[d]_{\left(s,1\right)} & TM\ar[d]^{Tq} & \ar@{}[d]|-{=} & X\myar{\left(u,v\right)}{r}\ar[d]_{\left(s,1\right)}\utwocell[0.3]{rd}{\gamma} & TM\ar[d]^{Tq}\\
TA\myard{Tp}{r}\utwocell[0.5]{ru}{\alpha} & TB & \; & TA\myard{Tp}{r}\ar[ur]|-{Tv\theta} & TB\utwocell[0.35]{ul}{T\textnormal{id}}
}
\]
universally factors as on the right above, where $\alpha$ and $\gamma$
are the respective morphisms of spans 
\[
\xymatrix@R=1em{ &  & TA\ar[ld]_{s}\ar[rd]^{p}\ar[dd]|-{\theta} &  &  &  & TA\ar[ld]_{s}\ar[rd]^{v\theta}\ar[dd]|-{\theta}\\
\alpha\colon & X &  & TB & \gamma\colon & X &  & TM\\
 &  & \bullet\ar[ur]_{qv}\ar[ul]^{u} &  &  &  & \bullet\ar[ur]_{v}\ar[ul]^{u}
}
\]
As all cells between generic morphisms are generic, it follows that
the category $\mathfrak{M}_{X}$ of generics out of $X$ is the slice
$\mathcal{E}/X$, and so for any $X\in\mathcal{E}$ we may take $P_{\left(-\right)}$
as the functor $\textnormal{dom}\colon\mathcal{E}/X\to\mathcal{E}$,
giving 
\[
\mathbf{Span}\left(\mathcal{E}\right)\left(X,T-\right)\cong\int_{\textnormal{lax}}^{m\in\mathcal{E}/X}\mathcal{E}\left(P_{m},-\right)
\]
Dual to the above, we see that $T\colon\mathcal{E}\to\mathbf{Span}\left(\mathcal{E}\right)^{\textnormal{co}}$
admits ``oplax-generic factorizations'' (meaning the same as lax-generic
factorizations except the direction of the 2-cells are reversed in
Definition \ref{deflaxgenericT}); indeed we may write 
\[
\mathbf{Span}\left(\mathcal{E}\right)^{\textnormal{co}}\left(X,T-\right)\cong\int_{\textnormal{oplax}}^{m\in\mathcal{E}/X}\mathcal{E}\left(P_{m},-\right)
\]
Moreover, the pseudofunctor $T\colon\mathcal{E}\to\mathbf{Span}_{\textnormal{iso}}\left(\mathcal{E}\right)$
admits both lax and oplax generic factorizations, as we may write
\[
\mathbf{Span}_{\textnormal{iso}}\left(\mathcal{E}\right)\left(X,T-\right)\cong\int_{\textnormal{lax}}^{m\in\left(\mathcal{E}/X\right)_{\textnormal{iso}}}\mathcal{E}\left(P_{m},-\right)\cong\int_{\textnormal{oplax}}^{m\in\left(\mathcal{E}/X\right)_{\textnormal{iso}}}\mathcal{E}\left(P_{m},-\right)
\]
where $\left(\mathcal{E}/X\right)_{\textnormal{iso}}$ contains the
objects of $\mathcal{E}/X$ and only those morphisms which are invertible.
The reader will also note that we do not have $\mathbf{Span}_{\textnormal{iso}}\left(\mathcal{E}\right)\left(X,T-\right)\simeq\sum_{\textnormal{ob }\mathcal{E}/X}\mathcal{E}\left(P_{m},-\right)$
as for each $T\in\mathcal{E},$ the right above is a discrete category,
but isomorphisms of spans are not unique (and so the canonical functor
$\mathbf{Span}_{\textnormal{iso}}\left(\mathcal{E}\right)\left(X,T-\right)\to\sum_{\textnormal{ob }\mathcal{E}/X}\mathcal{E}\left(P_{m},-\right)$
is not fully faithful). 
\end{example}

The case of spans is also interesting as it gives a simple example
in which generic factorizations are not unique in the sense that one
might initially expect. That is to say, given two generic factorizations
\[
\xymatrix@R=1em{X\ar[rd]_{\delta}\myar{f}{rr} & \; & TA &  & X\ar[rd]_{\delta}\myar{f}{rr} & \; & TA\\
 & TP\ar[ur]_{T\overline{f}}\utwocell[0.6]{u}{\alpha} &  &  &  & TP\ar[ur]_{T\overline{g}}\utwocell[0.6]{u}{\beta}
}
\]
(meaning isomorphisms $\alpha$ and $\beta$ as above), there is not
necessarily a coherent comparison 2-cell $\overline{f}\Rightarrow\overline{g}$.
\begin{example}
Consider a span
\[
\xymatrix{ & \mathbf{2}\ar[rd]^{\sigma}\ar[ld]_{!}\\
\mathbf{1} &  & \mathbf{2}
}
\]
where $\sigma$ is the swap map. Here we have the two distinct generic
factorizations
\[
\xymatrix@R=1em{\mathbf{1}\ar[rd]_{\left(!,1\right)}\myar{\left(!,\sigma\right)}{rr} & \; & T\mathbf{2} &  & \mathbf{1}\ar[rd]_{\left(!,1\right)}\myar{\left(!,\sigma\right)}{rr} & \; & T\mathbf{2}\\
 & T\mathbf{2}\ar[ur]_{T1}\utwocell[0.6]{u}{\sigma} &  &  &  & T\mathbf{2}\ar[ur]_{T\sigma}\utwocell[0.6]{u}{\textnormal{id}}
}
\]
\end{example}

In the following examples we will omit the verification that the generic
morphisms and cells are classified correctly, as these calculations
involving polynomial functors are quite technical.
\begin{example}
\label{polyinc} For a locally cartesian closed category $\mathcal{E}$,
one may form the bicategory of polynomials in $\mathcal{E}$, denoted
$\mathbf{Poly}\left(\mathcal{E}\right)$, whose objects are those
of $\mathcal{E}$, morphisms are triples $\left(s,p,t\right)\colon X\to Y$
as below
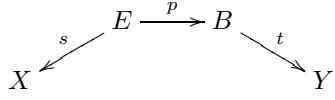
\[
\xymatrix@R=1em{ & E\ar[ld]_{s}\ar[r]^{p} & B\ar[rd]^{t}\\
X &  &  & Y
}
\]
called \emph{polynomials}, and a (general) 2-cell of polynomials is
a diagram as below
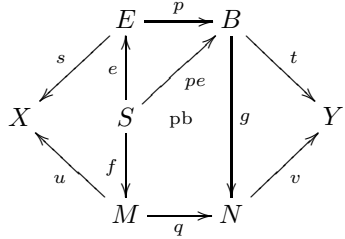
\[
\xymatrix{ & E\ar[dl]_{s}\ar[r]^{p} & B\ar[dr]^{t}\ar[dd]^{g}\\
X & S\ar[d]_{f}\ar[ur]_{pe}\ar[u]^{e}\ar@{}[r]|-{\textnormal{pb}} & \; & Y\\
 & M\ar[r]_{q}\ar[ul]^{u} & N\ar[ur]_{v}
}
\]
where the indicated square is a pullback\footnote{Actually a 2-cell is an equivalence class of such diagrams; or such
a diagram where the pullback is chosen.} \cite{gambinokock,weber}. If $e$ is invertible (so that the 2-cell
is just a pullback) the 2-cell is said to be \emph{cartesian} \cite{gambinokock,weber}. 

Similar to the case of spans, we see that the canonical pseudofunctor
$T\colon\mathcal{E}\to\mathbf{Poly}\left(\mathcal{E}\right)$ sending
a morphism $f$ to $\left(1,1,f\right)$ is lax familial. Indeed,
one can verify that a polynomial $X\nrightarrow TA$ is generic precisely
when it is isomorphic to the form 
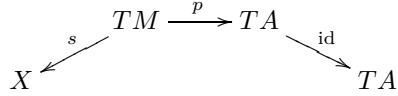
\[
\xymatrix@R=1em{ & TM\ar[ld]_{s}\ar[r]^{p} & TA\ar[rd]^{\textnormal{id}}\\
X &  &  & TA
}
\]
and it is trivial that any general 2-cell of polynomials as below
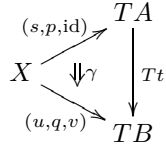
\[
\xymatrix@R=1em{ & TA\ar[dd]^{Tt}\\
X\ar[ur]^{\left(s,p,\textnormal{id}\right)}\ar[rd]_{\left(u,q,v\right)}\dtwocell[0.5]{r}{\gamma} & \;\\
 & TB
}
\]
is generic (as the domain of $T$ is a 1-category). Consequently,
denoting by $\Pi_{\mathcal{E}}$ the pseudomonad freely adding products
to a fibration \cite{fibredcategories}, we may take $P_{\left(-\right)}$
as the functor $\textnormal{pr}\colon\Pi_{\mathcal{E}}\left(\mathcal{E}/X\right)\to\mathcal{E}$
where $\Pi_{\mathcal{E}}\left(\mathcal{E}/X\right)$ is the category
with objects given by spans 
\[
\xymatrix@R=0.1em{X & T\ar[l]_{f}\ar[r]^{g} & U}
\]
out of $X$, and morphisms of spans from $\left(f,g\right)\nrightarrow\left(f',g'\right)$
given by a pair $\beta\colon U\to U'$ and $\alpha\colon W\to T$
(where $W$ is the chosen pullback of $\beta$ and $g'$) rendering
commutative the diagram 
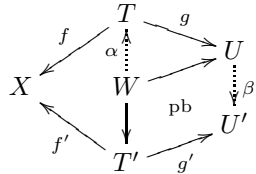
\[
\xymatrix@R=0.1em{ & T\ar[ldd]_{f}\ar[rd]^{g}\\
 &  & U\ar@{..>}[dd]^{\beta}\\
X & W\ar@{..>}[uu]^{\alpha}\ar[dd]\ar[ru]\ar@{}[rd]|-{\textnormal{pb}}\\
 &  & U'\\
 & T'\ar[luu]^{f'}\ar[ru]_{g'}
}
\]
As a consequence we have the formula
\[
\mathbf{Poly}\left(\mathcal{E}\right)\left(X,T-\right)\cong\int_{\textnormal{lax}}^{m\in\Pi_{\mathcal{E}}\left(\mathcal{E}/X\right)}\mathcal{E}\left(P_{m},-\right)
\]
for all $X\in\mathbf{Poly}\left(\mathcal{E}\right)$. 
\end{example}

\begin{rem}
From Example \ref{polyinc}, we see that the usual inclusion $\mathbf{Span}\left(\mathcal{E}\right)\to\mathbf{Poly}\left(\mathcal{E}\right)$
can be seen as coming from the unit components $u_{\mathcal{E}/X}\colon\mathcal{E}/X\to\Pi_{\mathcal{E}}\left(\mathcal{E}/X\right)$
of the pseudomonad $\Pi_{\mathcal{E}}$ for fibrations with products
\cite{fibredcategories}. Indeed, the family of functors $\mathbf{Span}\left(\mathcal{E}\right)\left(X,Y\right)\to\mathbf{Poly}\left(\mathcal{E}\right)\left(X,Y\right)$
may be written as the functors 
\[
\int_{\textnormal{lax}}^{m\in\mathcal{E}/X}\mathcal{E}\left(P_{m},Y\right)\to\int_{\textnormal{lax}}^{m\in\Pi_{\mathcal{E}}\left(\mathcal{E}/X\right)}\mathcal{E}\left(P_{m},Y\right)
\]
resulting from this reindexing.
\end{rem}

We now give a more complicated example, where $\mathscr{A}$ is not
a 1-category. In this situation the mixed lifting properties are necessary
(unlike the earlier examples where usual liftings would suffice),
and so it is no longer the case that every $\left(h,\gamma\right)$
out of a generic morphism is a generic 2-cell. 
\begin{example}
The canonical pseudofunctor $T\colon\mathbf{Span}\left(\mathcal{E}\right)^{\textnormal{co}}\to\mathbf{Poly}\left(\mathcal{E}\right)$,
given by
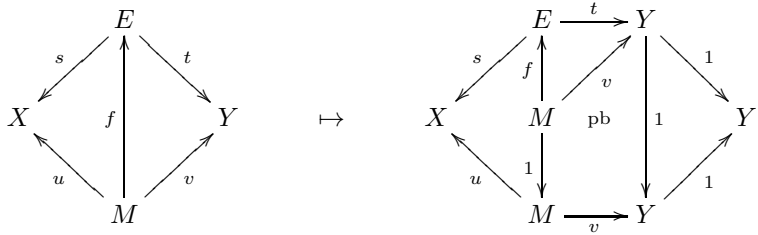
\[
\xymatrix{ & E\ar[dr]^{t}\ar[dl]_{s} &  &  &  & E\ar[dl]_{s}\ar[r]^{t} & Y\ar[dr]^{1}\ar[dd]^{1}\\
X &  & Y & \mapsto & X & M\ar[d]_{1}\ar[ur]_{v}\ar[u]^{f}\ar@{}[r]|-{\textnormal{pb}} & \; & Y\\
 & M\ar[ur]_{v}\ar[ul]^{u}\ar[uu]^{f} &  &  &  & M\ar[r]_{v}\ar[ul]^{u} & Y\ar[ur]_{1}
}
\]
is such that $T^{\textnormal{op}}$ is lax familial. Here a polynomial
$TA\nrightarrow X$ is opgeneric (meaning the opposite morphism in
$\mathbf{Poly}\left(\mathcal{E}\right)^{\textnormal{op}}$ is lax
generic with respect to $T^{\textnormal{op}}$) if it is isomorphic
to the form 
\[
\xymatrix@R=1em{ & TA\ar[ld]_{\textnormal{id}}\ar[r]^{\textnormal{id}} & TA\ar[rd]^{f}\\
TA &  &  & X
}
\]
and a pair $\left(\left(s,t\right),\gamma\right)$ out of a opgeneric
as below 
\[
\xymatrix@R=1em{ & X\ar[dd]^{T\left(s,t\right)=\left(s,t,1\right)}\\
TA\ar[ur]^{\left(1,1,f\right)}\ar[rd]_{\left(v,u,g\right)}\dltwocell[0.5]{r}{\gamma} & \;\\
 & TB
}
\]
is generic when $\gamma\colon\left(s,t,f\right)\Rightarrow\left(v,u,g\right)$
is a cartesian 2-cell of polynomials. We note also that cartesian
morphisms of polynomials are closed under vertical composition as
well as precomposition by another polynomial.

Given a general morphism of polynomials $\phi\colon\left(s,t,f\right)\Rightarrow\left(v,u,g\right)$
as given by the diagram 
\[
\xymatrix@R=1em{ & I\ar[ld]_{s}\ar[r]^{t} & M\ar[rd]^{f}\ar[dd]^{h}\\
TA & P\ar[u]^{e}\ar[ur]^{u'}\ar[d]_{h'}\ar@{}[rd]|-{\textnormal{pb}} & \; & X\\
 & J\ar[lu]^{v}\ar[r]_{u} & N\ar[ru]_{g}
}
\]
the opgeneric factorization of $\phi$ is given by 
\[
\xymatrix@R=1em{ &  &  &  &  & M\ar@/^{0.3pc}/[rd]^{\left(1,1,f\right)}\\
TA\ar@/^{1.4pc}/[rr]^{\left(s,t,f\right)}\ar@/_{1.4pc}/[rr]_{\left(v,u,g\right)}\dtwocell[0.5]{rr}{\phi} &  & X & = & TA\ar@/^{0.3pc}/[ur]^{T\left(s,t\right)}\ar@/_{0.3pc}/[rd]_{T\left(v,u\right)}\dtwocell[0.4]{r}{T\nu} & \;\dtwocell[0.5]{r}{\gamma} & X\\
 &  &  &  &  & N\ar@/_{0.3pc}/[ru]_{\left(1,1,g\right)}\ar[uu]|-{T\left(h,1\right)}
}
\]
where $\nu$ is the reversed morphism of spans on the left below 
\[
\xymatrix@R=1em{ & I\ar[rd]^{t}\ar[ld]_{s} &  &  &  & M\ar[ld]_{h}\ar[r]^{1}\ar[dd]_{h} & M\ar[rd]^{f}\ar[dd]^{h}\\
TA &  & M &  & N &  &  & X\\
 & P\ar[ur]_{u'}\ar[lu]^{vh'}\ar[uu]_{e} &  &  &  & N\ar[lu]^{1}\ar[r]_{1} & N\ar[ru]_{g}
}
\]
and $\gamma$ is the cartesian morphism of polynomials on the right
above. It follows that for any $X\in\mathcal{E}$ we may take $P_{\left(-\right)}$
as the functor 
\[
\xymatrix{\mathcal{E}/X\myar{\textnormal{dom}}{r} & \mathcal{E}\myar{\iota}{r} & \mathbf{Span}\left(\mathcal{E}\right)^{\textnormal{coop}}}
\]
where $\iota$ sends each morphism $h\colon A\to B$ to $\left(h,1_{A}\right)\in\mathbf{Span}\left(\mathcal{E}\right)^{\textnormal{coop}}$,
and get 
\[
\mathbf{Poly}\left(\mathcal{E}\right)^{\textnormal{op}}\left(X,T-\right)\cong\int_{\textnormal{lax}}^{m\in\mathcal{E}/X}\mathbf{Span}\left(\mathcal{E}\right)^{\textnormal{coop}}\left(P_{m},-\right).
\]
\end{example}

We now give a natural example which does not come from a pseudofunctor
of bicategories $T\colon\mathscr{A}\to\mathscr{B}$. Indeed, the following
may be seen as the main motivating example for this paper. 
\begin{example}
Following Carboni and Johnstone \cite{CarbJohnFam}, we consider the
pseudofunctor $\mathbf{Fam}\colon\mathbf{CAT}\to\mathbf{CAT}$ sending
a category $\mathcal{C}$ to the category $\mathbf{Fam}\left(\mathcal{C}\right)$
with objects given by families of objects of $\mathcal{C}$ denoted
$\left(A_{i}\in\mathcal{C}\colon i\in I\right)$, and morphisms $\left(A_{i}\in\mathcal{C}\colon i\in I\right)\nrightarrow\left(B_{j}\in\mathcal{C}\colon j\in J\right)$
given by a reindexing $\varphi\colon I\to J$ with comparison maps
$A_{i}\to B_{\varphi\left(i\right)}$ for each $i\in I$.

A trivial case of such a family is when $\mathcal{C}$ is any set
$I$, and $\left(i:i\in I\right)$ is such a family of objects (set-elements)
of $I$ indexed by the set $I$. Noting this, one can show the el-generic
objects of $\textnormal{el }\mathbf{Fam}$ are those elements of the
form $\left(I,\left(i:i\in I\right)\right)$ for a set $I$. And it
is clear that for any general element $\left(\mathcal{C},\left(B_{j}\colon j\in J\right)\right)$
of $\textnormal{el }\mathbf{Fam}$ that we have the ``generic factorization''
(that is an opcartesian map from a generic) 
\[
\xymatrix{\left(J,\left(j:j\in J\right)\right)\ar[rr]^{\left(B_{\left(-\right)},\textnormal{id}\right)} &  & \left(\mathcal{C},\left(B_{j}\colon j\in J\right)\right)}
\]
Also, a general morphism out of an el-generic object 
\[
\xymatrix{\left(I,\left(i:i\in I\right)\right)\ar[rr]^{\left(H_{\left(-\right)},\left(\varphi,\gamma\right)\right)} &  & \left(\mathcal{C},\left(B_{j}\colon j\in J\right)\right)}
\]
consists of a functor $H_{\left(-\right)}\colon I\to\mathcal{C}$,
a function $\varphi\colon I\to J$, and morphisms $\gamma_{i}\colon H_{i}\to B_{\varphi\left(i\right)}$
indexed over $i\in I$. Such a morphism is generic precisely when
every $\gamma_{i}$ is invertible.

It is then clear that the category of generic objects and generic
morphisms between them (note $H_{\left(-\right)}$ is uniquely determined
by $\varphi$ in this case) is isomorphic to $\mathbf{Set}$. It follows
that the $\mathbf{Fam}$ construction is given by\footnote{Whilst this example involves large categories, the indexing $\mathbf{Set}$
is locally small, and our results still apply. We do not wish to burden
this paper with a discussion of size conditions.} 
\[
\mathbf{Fam}\left(\mathcal{C}\right)=\int_{\textnormal{lax}}^{X\in\mathbf{Set}}\mathcal{C}^{X},\qquad\mathcal{C}\in\mathbf{CAT}
\]

It is worth noting that restricting to the category of finite sets
$\mathbf{Set}_{\textnormal{fin}}$, yields the finite families construction
$\mathbf{Fam}_{f}$, and restricting further the category of finite
sets and bijections $\mathbb{P}$ yields the free symmetric (strict)
monoidal category construction. 
\end{example}

The above shows that $\mathbf{Fam}$ is a lax conical colimit of representables.
However, it is also interesting that $\mathbf{Fam}$ is lax familial
in the sense of Definition \ref{deflaxfamilial}.
\begin{example}
The pseudofunctor $\mathbf{Fam}\colon\mathbf{CAT}\to\mathbf{CAT}$
is lax familial. Here the generic morphisms are those functors of
the form 
\[
\delta_{F}\colon\mathcal{C}\to\mathbf{Fam}\left(\textnormal{el }F\right)\colon X\mapsto\left(\left(X,x\right)\in\textnormal{el }F\colon x\in FX\right)
\]
for a presheaf $F\colon\mathcal{C}\to\mathbf{Set}$ (Weber refers
to these as ``functors endowing $\mathcal{C}$ with elements'' \cite[Definition 5.10]{Webfam}).
A cell out of such a generic morphism 
\[
\xymatrix@R=1em{ & \mathbf{Fam}\left(\textnormal{el }F\right)\ar[dd]^{\mathbf{Fam}\left(H\right)}\\
\mathcal{C}\ar[ur]^{\delta}\dltwocell[0.5]{r}{\gamma}\ar[rd]_{z} & \;\\
 & \mathbf{Fam}\left(\mathcal{B}\right)
}
\]
is generic when the comparison maps (not necessarily the reindexing
maps) comprising each $\gamma_{X}$ for $X\in\mathcal{C}$ are required
invertible. It follows that the lax familial representatibly of $\mathbf{Fam}$
is shown by the formula
\[
\mathbf{CAT}\left(\mathcal{C},\mathbf{Fam}\left(-\right)\right)\cong\int_{\textnormal{lax}}^{F\colon\mathcal{C}\to\mathbf{Set}}\mathbf{CAT}\left(\textnormal{el }F,-\right)
\]
for each $\mathcal{C}\in\mathbf{CAT}$. 
\end{example}

\section{The spectrum factorization of a lax familial pseudofunctor\label{spectrumsec}}

In the simpler dimension-one case, Diers \cite{Diers} showed that
familial functors (as defined in Definition \ref{leftmultiadjoint})
have the following simple characterization:
\begin{thm}
[Diers] Let $T\colon\mathcal{A}\to\mathcal{B}$ be a functor of categories.
Then the following are equivalent: 
\begin{enumerate}
\item the functor $T$ is familial; 
\item there exists a factorization
\[
\xymatrix@R=1em{\mathcal{A}\ar[rr]^{T}\ar[rd]_{G} &  & \mathcal{B}\\
 & \mathcal{M}\ar[ur]_{V}
}
\]
such that: 
\begin{enumerate}
\item $V$ is a discrete fibration;
\item $G$ has a left adjoint.
\end{enumerate}
\end{enumerate}
\end{thm}

When $\mathcal{A}$ has a terminal object, it is not hard to see that
$\mathcal{M}\simeq\mathcal{B}/T\mathbf{1}$. This gives the following
simple consequence:
\begin{cor}
Let $T\colon\mathcal{A}\to\mathcal{B}$ be a functor of categories,
and assume $\mathcal{A}$ has a terminal object. Then $T$ is familial
(is a parametric right adjoint) if and only if the canonical projection
\[
T_{1}\colon\mathcal{A}/\mathbf{1}\to\mathcal{B}/T\mathbf{1}
\]
has a left adjoint. 
\end{cor}

It is the purpose of this section to find an analogue of these results
in the dimension two case. However, as we will see, this is much more
complicated than simply asking for a left bi-adjoint. Instead we will
require certain types of ``lax'' adjunctions (or adjunctions up
to adjunction).

\subsection{Lax $\mathsf{F}$-adjunctions}

In the setting of an adjunction of functors $F\dashv G\colon\mathcal{A}\to\mathcal{M}$
we have natural hom-set isomorphisms $\mathcal{A}\left(F_{m},A\right)\cong\mathcal{M}\left(m,GA\right)$.
More generally, one can talk about bi-adjunctions of pseudofunctors
$F\dashv G\colon\mathscr{A}\to\mathscr{M}$ where we only ask for
natural hom-category equivalences $\mathscr{A}\left(F_{m},A\right)\simeq\mathscr{M}\left(m,GA\right)$
\cite{biequivtri}. However, even this notion is often too strong. 

Central to the theory of lax familial pseudofunctors is the theory
of lax adjunctions \cite{bettipower}, where one only asks that we
have adjoint pairs
\[
L_{m,A}\colon\mathscr{A}\left(F_{m},A\right)\to\mathscr{M}\left(m,GA\right),\qquad R_{m,A}\colon\mathscr{M}\left(m,GA\right)\to\mathscr{A}\left(F_{m},A\right)
\]
pseudonatural (or even lax natural) in $A\in\mathscr{A}$ and $m\in\mathscr{M}$. 

The following type of lax adjunctions, called \emph{lax $\mathsf{F}$-adjunctions},
appear when studying familial pseudofunctors.\footnote{Here the $\mathsf{F}$ denotes the category whose objects are fully
faithful functors and morphisms are pseudo-commuting squares. Moreover,
these concepts arise from considering $\mathsf{F}$-enriched (bi)categories;
though we will not use this enrichment perspective \cite{Fcaty}.} These are the lax adjunctions which naturally restrict to biadjunctions
on a class of ``tight'' maps. A simple example of tight maps are
the pseudo-commuting triangles (pseudo slice) of the lax slice category
of a pseudofunctor. Before defining lax $\mathsf{F}$-adjunctions,
we must first define $\mathsf{F}$-bicategories and see how they assemble
into a tricategory $\mathsf{F}\textnormal{-}\mathbf{Bicat}$. It is
not hard to verify this data forms a tricategory given that bicategories,
pseudofunctors, pseudo-natural transformations and modifications do
\cite{BicatTriequiv}.
\begin{defn}
The following notions below:
\begin{itemize}
\item an \emph{$\mathsf{F}$-bicategory} is a bicategory $\mathscr{A}$
equipped with an identity on objects, injective on 1-cells, locally
fully faithful functor $\mathscr{A}_{\mathsf{T}}\to\mathscr{A}$.
The 1-cells of $\mathscr{A}_{\mathsf{T}}$ are called the \emph{tight}
1-cells of $\mathscr{A}$ and are required to be closed under invertible
2-cells;
\item an \emph{$\mathsf{F}$-pseudofunctor} $\left(\mathscr{A},\mathscr{A}_{T}\right)\nrightarrow\left(\mathscr{B},\mathscr{B}_{T}\right)$
is a pseudofunctor $F\colon\mathscr{A}\to\mathscr{B}$ which restricts
to a pseudofunctor $F_{T}\colon\mathscr{A}_{T}\to\mathscr{B}_{T}$;
\item a \emph{lax $\mathsf{F}$-natural transformation} $\alpha\colon F\Rightarrow G\colon\left(\mathscr{A},\mathscr{A}_{T}\right)\to\left(\mathscr{B},\mathscr{B}_{T}\right)$
is a lax natural transformation $\alpha\colon F\Rightarrow G$ such
that both:
\begin{enumerate}
\item for all $X\in\mathscr{A}$, $\alpha_{X}\colon FX\to GX$ is tight;
\item for all $f\colon X\to Y$ tight, $\alpha_{f}\colon Gf\cdot\alpha_{X}\Rightarrow\alpha_{Y}\cdot Ff$
is invertible.
\end{enumerate}
\end{itemize}
define the tricategory $\mathsf{F}\textnormal{-}\mathbf{Bicat}$ of
$\mathsf{F}$-bicategories, $\mathsf{F}$-pseudofunctors, lax $\mathsf{F}$-natural
transformations, and modifications.
\end{defn}

The above allows for a particularly simple definition of lax $\mathsf{F}$-adjunctions.

\begin{defn}[Lax $\mathsf{F}$-adjunction]
\label{laxFadj} A \emph{lax $\mathsf{F}$-adjunction} of $\mathsf{F}$-pseudofunctors
\[
\xymatrix{\left(\mathscr{A},\mathscr{A}_{T}\right)\ar@<1.4ex>[rr]^{F} & \perp & \left(\mathscr{B},\mathscr{B}_{T}\right)\ar@<1.2ex>[ll]^{G}}
\]
is a biadjunction \cite{biequivtri} in the tricategory $\mathsf{F}\textnormal{-}\mathbf{Bicat}$.
\end{defn}

\begin{rem}
It is worth noting that this definition immediately tells us that
lax $\mathsf{F}$-adjunctions enjoy nice properties such as uniqueness
of adjoints up to equivalence.
\end{rem}

Whilst the above definition is conceptually informative, for our purposes
it will be more useful to define these adjunctions in terms of universal
arrows. This is due to the connection between the universal arrow
definition and notions of genericity. 
\begin{rem}
From now on we will regard the right adjoint $G$ as a $\mathsf{F}$-pseudofunctor
$G\colon\left(\mathscr{A},\mathscr{A}_{T}\right)\to\left(\mathscr{M},\mathscr{M}_{T}\right)$
to more closely match the notation we will use use later on.
\end{rem}

The characterization of lax $\mathsf{F}$-adjunctions by universal
arrows is slightly technical, so we will break it into parts.
\begin{defn}
[Universal pair] Given an pseudofunctor $G\colon\mathscr{A}\to\mathscr{M}$,
object $F_{m}$ in $\mathscr{M}$, and a diagram 
\[
\xymatrix@R=0.6em{m\myar{f}{rr}\ar[rdd]_{\eta_{m}} & \; & GA\\
\\
 & GF_{m}\ar[uur]_{G\overline{f}}\utwocell[0.6]{uu}{\gamma_{f}}
}
\]
we say the pair $\left(\overline{f},\gamma_{f}\right)$ is \emph{universal}
if for any $\overline{g}\colon F_{m}\to A$ and 2-cell $\beta$ as
below
\[
\xymatrix@R=0.6em{m\myar{f}{rr}\ar[rdd]_{\eta_{m}} & \; & GA &  & m\myar{f}{rr}\ar[rdd]_{\eta_{m}} & \; & GA\\
 &  &  & =\\
 & GF_{m}\ar[uur]_{G\overline{g}}\utwocell[0.6]{uu}{\beta} &  &  &  & GF_{m}\ar[uur]^{G\overline{f}}\utwocell[0.6]{uu}{\gamma_{f}}\ar@/_{1.3pc}/[uur]_{G\overline{g}} & \ultwocell[0.4]{uul}{G\widetilde{\beta}}
}
\]
there exists a unique $\widetilde{\beta}\colon\overline{g}\Rightarrow\overline{f}$
such that the above equality holds.
\end{defn}

The following defines what one should think of as the ``universal
arrows'' of a lax $\mathsf{F}$-adjunction.

\begin{defn}
[$\mathsf{F}$-universal arrows] Given an $\mathsf{F}$-pseudofunctor
$G\colon\left(\mathscr{A},\mathscr{A}_{T}\right)\to\left(\mathscr{M},\mathscr{M}_{T}\right)$,
we say a morphism $\eta_{m}\colon m\to GF_{m}$ (where $F_{m}$ is
some object of $\mathscr{M})$ is \emph{universal }if for any 1-cell
$f\colon m\to GA$ there exists a $\overline{f}\colon F_{m}\to A$
and a 2-cell
\[
\xymatrix@R=0.6em{m\myar{f}{rr}\ar[rdd]_{\eta_{m}} & \; & GA\\
\\
 & GF_{m}\ar[uur]_{G\overline{f}}\utwocell[0.6]{uu}{\gamma_{f}}
}
\]
such that the pair $\left(\overline{f},\gamma_{f}\right)$ is universal.
We say that $\eta_{m}$ is $\mathsf{F}$-universal if in addition\footnote{The reader will of course notice that such a $\eta_{m}$ is unique
up to equivalence.}
\begin{enumerate}
\item[(i)]  the 1-cell $\eta_{m}$ is tight;
\item[(ii)]  for every tight 1-cell $f\colon m\to GA$ in $\mathscr{M}$, the
2-cell $\gamma_{f}$ is invertible and $\overline{f}\colon F_{m}\to A$
is tight;
\item[(iii)]  the diagram
\[
\xymatrix@R=0.6em{m\myar{\eta_{m}}{rr}\ar[rdd]_{\eta_{m}} & \; & GF_{m}\\
\\
 & GF_{m}\ar[uur]_{G1_{F_{m}}}\utwocell[0.6]{uu}{\textnormal{id}}
}
\]
exhibits $\left(1_{F_{m}},\textnormal{id}\right)$ as a universal
pair;
\item[(iv)]  for any universal pair $\left(\overline{f},\gamma_{f}\right)$,
the $G$-whiskering by a tight $g\colon A\to B$ 
\[
\xymatrix@R=0.6em{m\myar{f}{rr}\ar[rdd]_{\eta_{m}} & \; & GA\ar[r]^{Gg} & GB\\
\\
 & GF_{m}\ar[uur]_{G\overline{f}}\utwocell[0.6]{uu}{\gamma_{f}}\ar@/_{1pc}/[rruu]_{Gg\overline{f}} & \ar@{}[uu]|-{\cong}
}
\]
exhibits $\left(g\overline{f},Gg\cdot\gamma_{f}\right)$ as a universal
pair.
\end{enumerate}
\end{defn}

The following proposition makes precise the characterization of a
lax $\mathsf{F}$-adjunction in terms of universal arrows. 
\begin{prop}
\label{unilaxFadj} Given an $\mathsf{F}$-pseudofunctor $G\colon\left(\mathscr{A},\mathscr{A}_{T}\right)\to\left(\mathscr{M},\mathscr{M}_{T}\right)$,
$G$ has a left lax $\mathsf{F}$-adjoint as in Definition \ref{laxFadj}
if and only if both:
\begin{enumerate}
\item for every object $m$ in $\mathscr{M}$, there exists a $\mathsf{F}$-universal
1-cell $\eta_{m}\colon m\to GA$;
\item for all 1-cells $\mu$ and $\nu$ as below, $\overline{\eta_{k}\nu}\cdot\overline{\eta_{n}\mu}$
equipped with the 2-cell
\[
\xymatrix{m\myar{\eta_{m}}{rr}\ar[d]_{\mu}\dltwocell[0.5]{drr}{\gamma_{\eta_{n}\mu}} &  & GF_{m}\ar[d]|-{G\left(\overline{\eta_{n}\mu}\right)}\ar@/^{3pc}/[dd]^{G\left(\overline{\eta_{k}\nu}\cdot\overline{\eta_{n}\mu}\right)}\\
n\myar{\eta_{n}}{rr}\ar[d]_{\nu}\dltwocell[0.5]{drr}{\gamma_{\eta_{k}\nu}} &  & GF_{n}\ar[d]|-{G\left(\overline{\eta_{k}\nu}\right)}\ar@{}[r]|-{\cong\;\;} & \;\\
k\myard{\eta_{k}}{rr} &  & GF_{k}
}
\]
is universal.
\end{enumerate}
\end{prop}

We will not give all the technical details of the proof, but the following
remark should convince the reader of this characterization.
\begin{rem}
It comes for free that for all $A\in\mathscr{A}$, the universal pair
\[
\xymatrix@R=0.6em{GA\myar{1_{GA}}{rr}\ar[rdd]_{\eta_{GA}} & \; & GA\\
\\
 & GF_{GA}\ar[uur]_{G\overline{\textnormal{id}}}\utwocell[0.6]{uu}{\gamma_{1_{GA}}}
}
\]
has the 2-cell component $\gamma_{1_{GA}}$ invertible (as identity
1-cells are necessarily tight). This is one of the triangle identities.
The other triangle identity which asks for the composite of $F_{\eta_{m}}$
and $\epsilon_{F_{m}}$ constructed as below
\[
\xymatrix{m\myar{\eta_{m}}{rr}\ar[d]_{\eta_{m}}\dltwocell[0.5]{drr}{\gamma_{\eta_{GF_{m}}\eta_{m}}} &  & GF_{m}\ar[d]|-{GF_{\eta_{m}}}\\
GF_{m}\myar{\eta_{GF_{m}}}{rr}\dltwocell[0.5]{drr}{\gamma_{1_{GF_{m}}}}\ar@/_{1pc}/[rrd]_{1_{GF_{m}}} &  & GFGF_{m}\ar[d]|-{G\epsilon_{Fm}}\\
 &  & GF_{m}
}
\]
to be isomorphic to the identity, is equivalent to (iii) in the presence
of (iv). Pseudofunctoriality of $F$ is clear from (2) and (iii).

The reader will also recognize that $L_{m,A}$ and $R_{m,A}$ are
pseudonatural in $A\in\mathscr{A}$ and $m\in\mathscr{M}$ respectively;
and also pseudonatural in $m\in\mathscr{M}_{T}$ and $A\in\mathscr{A}_{T}$
respectively. Indeed, $L_{m,A}\colon\mathscr{A}\left(F_{m},A\right)\to\mathscr{M}\left(m,GA\right)$
is defined by applying $G$ and composing with $\eta_{m}$, and $R_{m,A}\colon\mathscr{M}\left(m,GA\right)\to\mathscr{A}\left(F_{m},A\right)$
is defined by applying $F$ and composing with $\epsilon_{A}$. Also,
it is not hard to see that $\eta$ and $\epsilon$ become lax $\mathsf{F}$-natural
transformations given the universal arrow viewpoint. Finally, it is
worth noting that each $\gamma$ is invertible if and only if the
unit $\eta$ is pseudonatural.
\end{rem}

The following property of lax $\mathsf{F}$-adjunctions, that the
operations $\widetilde{\left(-\right)}$ respect isomorphisms, will
be useful later in this section.
\begin{lem}
\label{tildepresiso} Given a pseudofunctor $G\colon\mathscr{A}\to\mathscr{M}$
with a left lax $\mathsf{F}$-adjoint $\left(F,\eta,\gamma\right)$,
the operation $\beta\mapsto\widetilde{\beta}$ respects isomorphisms
on tight maps.
\end{lem}

\begin{proof}
Suppose we have an equality as below where $\overline{g}\colon F_{m}\to A$
is tight
\[
\xymatrix@R=0.6em{m\myar{f}{rr}\ar[rdd]_{\eta_{m}} & \; & GA &  & m\myar{f}{rr}\ar[rdd]_{\eta_{m}} & \; & GA\\
 &  &  & =\\
 & GF_{m}\ar[uur]_{G\overline{g}}\utwocell[0.6]{uu}{\beta} &  &  &  & GF_{m}\ar[uur]^{G\overline{f}}\utwocell[0.6]{uu}{\gamma_{f}}\ar@/_{1.3pc}/[uur]_{G\overline{g}} & \ultwocell[0.35]{uul}{Gb}
}
\]
and suppose further that $\beta$ has an inverse, so that we have
the equality
\[
\xymatrix@R=0.6em{m\ar[r]^{\eta_{m}}\ar@/_{1.4pc}/[rr]|-{f}\ar[rdd]_{\eta_{m}} & GF_{m}\ar[r]^{G\overline{g}} & GA &  & m\ar[r]^{\eta_{m}}\ar[rdd]_{\eta_{m}} & GF_{m}\ar[r]^{G\overline{g}} & GA\\
 & \utwocell[0.5]{u}{\beta^{-1}} &  & =\\
 & GF_{m}\ar[uur]_{G\overline{f}}\utwocell[0.6]{u}{\gamma_{f}} &  &  &  & GF_{m}\ar[uur]^{G\overline{g}}\utwocell[0.6]{uu}{\textnormal{id}}\ar@/_{1.3pc}/[uur]_{G\overline{f}} & \ultwocell[0.35]{uul}{Ga}
}
\]
where we have used axioms (iii) and (iv) to realize the pair consisting
of the identity 2-cell (given on the right above) and $\overline{g}$
as universal. It is then straightforward to verify $a$ is inverse
to $b$.
\end{proof}
\begin{rem}
It is not hard to see that in the presence of axiom (iv), the above
lemma is equivalent to (iii).
\end{rem}

\subsection{Factoring through the spectrum}

We now have the necessary background on lax adjunctions, and can move
towards understanding how a lax familial pseudofunctor factors through
the spectrum. This will only require the following simple lemma.
\begin{lem}
\label{cart2cell} Suppose $V\colon\mathscr{M}\to\mathscr{B}$ is
a locally discrete fibration of bicategories. Then given any 2-cell
$\alpha\colon f\Rightarrow g\colon X\to Vm$ as on the right below
\[
\xymatrix@R=1em{f^{*}m\ar[rrd]^{f_{c}}\ar[dd]_{\hat{\alpha}} &  &  &  & X\ar[rrd]^{f}\ar[dd]_{\textnormal{id}}\\
\dtwocell[0.4]{rr}{\overline{\textnormal{\ensuremath{\alpha}}}} &  & m & \mapsto & \dtwocell[0.4]{rr}{\textnormal{\ensuremath{\alpha}}} &  & Vm\\
g^{*}m\ar[rru]_{g_{c}} &  &  &  & X\ar[rru]_{g}
}
\]
with cartesian lifts $f_{c}$ and $g_{c}$ of $f$ and $g$, there
exists a unique pair $\left(\hat{\alpha},\overline{\alpha}\right)$
as on the left above which is sent to $\alpha$ by $V$. Moreover,
if $\alpha$ is invertible then both $\hat{\alpha}$ and $\overline{\alpha}$
are.
\end{lem}

\begin{proof}
Suppose without loss of generality that $V$ is the projection $\int_{\textnormal{lax}}^{B\in\mathscr{B}}FB\to\mathscr{B}$
for a pseudofunctor $F\colon\mathscr{B}^{\textnormal{op}}\to\mathbf{Cat}$.
Then we may construct a diagram as on the left below
\[
\xymatrix@R=1em{\left(X,a\right)\ar[rrd]^{\left(f,\cong\right)}\ar[dd]_{\left(1,\lambda\right)} &  &  &  & X\ar[rrd]^{f}\ar[dd]_{\textnormal{id}}\\
\dtwocell[0.4]{rr}{\alpha} &  & \left(Y,m\right) & \mapsto & \dtwocell[0.4]{rr}{\textnormal{\ensuremath{\alpha}}} &  & V\left(Y,m\right)\\
\left(X,b\right)\ar[rru]_{\left(g,\cong\right)} &  &  &  & X\ar[rru]_{g}
}
\]
where $\lambda$ is the unique map such that
\[
\xymatrix{a\myar{\cong}{r} & Ff\left(m\right)\myar{\left(F\alpha\right)_{m}}{r} & Fg\left(m\right) & = & a\myar{\lambda}{r} & b\myar{\cong}{r} & Fg\left(m\right)}
\]
holds. It is clear this is the only choice of $\lambda$, and that
if $\alpha$ is invertible then so is $\lambda$.
\end{proof}
\begin{rem}
There should be an analogue of the above without assuming $V$ to
be locally discrete, so that $V$ is the projection $\int_{\textnormal{lax}}^{B\in\mathscr{B}}FB\to\mathscr{B}$
for a trifunctor $F\colon\mathscr{B}^{\textnormal{op}}\to\mathbf{Bicat}$.
However, this is beyond the scope of this paper. 
\end{rem}

We can now prove the main result of this section, which provides a
conceptually nice description of lax familial pseudofunctors. Recall
also that the tight maps of a bicategory are something we must specify,
and not part of the data of the original bicategory.
\begin{thm}
[Spectrum factorization]\label{specfacthm} Let $T\colon\mathscr{A}\to\mathscr{B}$
be a pseudofunctor of bicategories. Then the following are equivalent: 
\begin{enumerate}
\item the pseudofunctor $T$ is lax familial; 
\item there exists a factorization
\[
\xymatrix@R=1em{\mathscr{A}\ar[rr]^{T}\ar[rd]_{G} &  & \mathscr{B}\\
 & \mathscr{M}\ar[ur]_{V}
}
\]
such that: 
\begin{enumerate}
\item $V$ is a locally discrete fibration of bicategories;
\item $G$ has a left lax $\mathsf{F}$-adjoint (where all 1-cells in $\mathscr{A}$
are tight and the $V$-cartesian 1-cells of $\mathscr{M}$ are tight).
\end{enumerate}
\end{enumerate}
\end{thm}

\begin{proof}
$\left(2\right)\Rightarrow\left(1\right)\colon$ We first note that
for any $f\colon X\to TA$ in $\mathscr{B}$, we have a cartesian
lift $f_{c}\colon m\to GA$ in $\mathscr{M}$. We thus have an assignment
\[
\xymatrix@R=1em{m\ar[ddr]_{\eta_{m}}\ar[rr]^{f_{c}} & \; & GA &  & X\ar[ddr]_{\delta_{m}}\ar[rr]^{f} & \; & TA\\
 &  &  & \mapsto\\
 & GF_{m}\ar[uur]_{G\overline{f_{c}}}\utwocell[0.6]{uu}{\gamma_{f_{c}}} &  &  &  & TF_{m}\ar[uur]_{T\overline{f_{c}}}\utwocell[0.6]{uu}{V\gamma_{f_{c}}}
}
\]
and as $\gamma$ is invertible on cartesian maps, this is a factorization
of $f$. We thus need only check that each $\delta_{m}$ is lax-generic,
and that generic 2-cells compose.

Consider now a 2-cell $\alpha$ as on the right below
\[
\xymatrix@R=1em{n\ar[dd]_{\eta_{n}}\ar[r]^{\hat{\alpha}} & m\ar[r]^{f_{c}} & GA\ar[dd]^{Gk} &  & X\ar[dd]_{\delta_{n}}\ar[rr]^{f} &  & TA\ar[dd]^{Tk}\\
 &  &  & \mapsto\\
GF_{n}\ar[rr]_{Gh}\utwocell[0.5]{rruu}{\overline{\alpha}} &  & GC &  & TF_{n}\ar[rr]_{Th}\utwocell[0.5]{rruu}{\alpha} &  & TC
}
\]
and its unique preimage as on the left above given by Lemma \ref{cart2cell}.
This $\overline{\alpha}$ in turn has a factorization as on the left
below
\[
\xymatrix@R=1em{n\ar[dd]_{\eta_{n}}\ar[r]^{\hat{\alpha}} & m\ar[r]^{f_{c}} & GA\ar[dd]^{Gk} &  & X\ar[dd]_{\eta_{n}}\ar[rr]^{f} &  & TA\ar[dd]^{Tk}\\
 & \utwocell[0.35]{ul}{\gamma_{f_{c}\hat{\alpha}}} &  & \mapsto &  & \utwocell[0.45]{ul}{V\gamma_{f_{c}\hat{\alpha}}}\\
GF_{n}\ar[rr]_{Gh}\ar[uurr]_{G\overline{f_{c}\hat{\alpha}}} &  & GC\utwocell[0.5]{ul}{G\xi} &  & TF_{n}\ar[rr]_{Th}\ar[uurr]_{T\overline{f_{c}\hat{\alpha}}} &  & TC\utwocell[0.5]{ul}{T\xi}
}
\]
since universality of $\left(f_{c}\hat{\alpha},\gamma_{f_{c}\hat{\alpha}}\right)$
is preserved by $Gk$, thus giving a factorization of $\alpha$ as
on the right above. Note that if $\alpha$, and hence $\hat{\alpha}$
and $\overline{\alpha}$ are invertible, then $\gamma_{f_{c}\hat{\alpha}}$
is invertible (as it is on all cartesian 1-cells), and $\xi$ is invertible
by Lemma \ref{tildepresiso}.

Given another factorization as on the right below, we can lift $\sigma$
by Lemma \ref{cart2cell}
\[
\xymatrix@R=1em{n\ar[dd]_{\eta_{n}}\ar[r]^{\hat{\sigma}} & m\ar[r]^{f_{c}} & GA\ar[dd]^{Gk} &  & X\ar[dd]_{\eta_{n}}\ar[rr]^{f} &  & TA\ar[dd]^{Tk}\\
 & \utwocell[0.35]{ul}{\overline{\sigma}} &  & \mapsto &  & \utwocell[0.45]{ul}{\sigma}\\
GF_{n}\ar[rr]_{Gh}\ar[uurr]_{G\overline{g}} &  & GC\utwocell[0.5]{ul}{G\varphi} &  & TF_{n}\ar[rr]_{Th}\ar[uurr]_{T\overline{g}} &  & TC\utwocell[0.6]{ul}{T\varphi}
}
\]
giving the left above. Noting that $\hat{\sigma}=\hat{\alpha}$ and
that the left pasting above is $\overline{\alpha}$ by uniqueness,
we can then factor $\overline{\sigma}$ through $\gamma_{f_{c}\hat{\alpha}}$
recovering a comparison map $\psi\colon\overline{g}\Rightarrow\overline{f_{c}\hat{\alpha}}$
satisfying the required conditions. The subterminality of each $V\gamma_{f_{c}\hat{\alpha}}$
stems from the uniqueness of factorizations through $\gamma_{f_{c}\hat{\alpha}}$.

Finally, to see that generic cells compose, observe that a cell as
on the right below
\[
\xymatrix@R=1em{n\ar[dd]_{\hat{\gamma}}\ar[rr]^{\eta_{n}} &  & GF_{n}\ar[dd]^{Gh} &  &  &  & TF_{n}\ar[dd]^{Th}\\
 &  & \dtwocell[0.5]{ll}{\overline{\gamma}} & \mapsto & X\ar[rrd]_{z}\ar[rru]^{\delta_{n}} &  & \dtwocell[0.4]{ll}{\gamma}\\
m\ar[rr]_{z_{c}} &  & GC &  &  &  & TC
}
\]
is generic precisely when its lift as on the left above, given by
Lemma \ref{cart2cell}, exhibits $\left(h,\overline{\gamma}\right)$
as a universal pair. Also observe that every generic is of the form
$\delta_{n}$, since given any generic $\delta$ and cartesian lift
$\delta_{c}$ we have an isomorphism
\[
\xymatrix@R=1em{m\ar[ddr]_{\eta_{m}}\ar[rr]^{\delta_{c}} & \; & GA &  & X\ar[ddr]_{\delta_{m}}\ar[rr]^{\delta} & \; & TA\\
 &  &  & \mapsto\\
 & GF_{m}\ar[uur]_{G\overline{\delta_{c}}}\utwocell[0.6]{uu}{\gamma_{\delta_{c}}} &  &  &  & TF_{m}\ar[uur]_{T\overline{\delta_{c}}}\utwocell[0.6]{uu}{V\gamma_{\delta_{c}}}
}
\]
and we know that $\left(\overline{\delta_{c}},V\gamma_{\delta_{c}}\right)$
is an equivalence by Lemma \ref{opcartequiv}. It follows that two
generic cells as on the right below
\[
\xymatrix@R=1em{n\ar[d]_{\hat{\gamma}}\ar[rr]^{\eta_{n}} &  & GF_{n}\ar[dd]^{Gh} &  & X\ar[rr]^{\delta_{n}}\ar[dd]_{f} &  & TF_{n}\ar[dd]^{Th}\\
\bullet\ar[d]_{f_{c}} &  & \dtwocell[0.5]{ll}{\overline{\gamma}} &  &  &  & \dtwocell[0.5]{ll}{\gamma}\\
m\ar[rr]_{\eta_{m}}\ar[d]_{\hat{\phi}} &  & GF_{m}\ar[dd]^{Gk} & \mapsto & Y\ar[rr]^{\delta_{m}}\ar[dd]_{g} &  & TF_{m}\ar[dd]^{Tk}\\
\bullet\ar[d]_{g_{c}} &  & \dtwocell[0.5]{ll}{\overline{\phi}} &  &  &  & \dtwocell[0.5]{ll}{\phi}\\
w\ar[rr]_{\eta_{w}} &  & GF_{w} &  & Z\ar[rr]_{\delta_{w}} &  & TF_{w}
}
\]
compose to a generic, as the composite on the left above is universal.

$\left(1\right)\Rightarrow\left(2\right)\colon$ Supposing that $T\colon\mathscr{A}\to\mathscr{B}$
is lax familial, we may construct the spectrum $\mathfrak{M}_{\left(-\right)}\colon\mathscr{B}^{\textnormal{op}}\to\mathbf{Cat}$
as in Lemma \ref{constructspectrum} and factor $T$ as 
\[
\xymatrix@R=1em{\mathscr{A}\myar{G}{r} & {\displaystyle \int_{\textnormal{lax}}^{X\in\mathscr{B}}\mathfrak{M}_{\left(-\right)}}\myar{V}{r} & \mathscr{B}}
\]
where $G$ sends each $A\in\mathscr{A}$ to $TA\in\mathscr{B}$ with
the generic morphism $\delta_{A}\colon TA\to T\overline{A}$ being
part of the generic factorization
\[
\xymatrix@R=1em{TA\ar[r]^{\delta_{A}} & T\overline{A}\ar[r]^{Te_{A}} & TA}
\]
of the identity. We choose all morphisms of $\mathscr{A}$ to be tight,
and the cartesian morphisms against the projection $V$ to be tight.
A 1-cell $h\colon A\to B$ in $\mathscr{A}$ is sent to $Th$ with
the pair $\left(\overline{h},\cong\right)$ comprising the left side
\[
\xymatrix@R=1em{TA\ar[r]^{\delta_{A}}\ar[d]_{Th}\ar@{}[rd]|-{\cong} & T\overline{A}\ar[r]^{Te_{A}}\ar[d]^{T\overline{h}}\ar@{}[rd]|-{\cong} & TA\ar[d]^{Th}\\
TB\ar[r]_{\delta_{B}} & T\overline{B}\ar[r]_{Te_{B}} & TB
}
\]
of the generic factorization above. A given 2-cell $\lambda\colon h\Rightarrow k$
is sent to $T\lambda\colon Th\Rightarrow Tk$, which satisfies 
\[
\xymatrix@R=1em{TA\ar@/_{0.7pc}/[dd]_{Tk}\ar@/^{0.7pc}/[dd]^{Th}\ar[r]^{\delta_{A}} & T\overline{A}\ar[dd]^{T\overline{h}}\ar@{}[ldd]|-{\quad\cong} &  & TA\ar[dd]_{Tk}\ar[r]^{\delta_{A}}\ar@{}[rdd]|-{\cong\quad} & T\overline{A}\ar@/_{0.7pc}/[dd]_{T\overline{k}}\ar@/^{0.7pc}/[dd]^{T\overline{h}}\\
\ltwocell[0]{}{T\lambda} &  & = &  & \ltwocell[0]{}{T\overline{\lambda}}\\
TB\ar[r]_{\delta_{B}} & T\overline{B} &  & TB\ar[r]_{\delta_{B}} & T\overline{B}
}
\]
for some (necessarily unique) $\overline{\lambda}\colon\overline{h}\Rightarrow\overline{k}$.
To see this, note that the left diagram has a generic factorization
\[
\xymatrix@R=1em{TA\ar@/_{0.7pc}/[dd]_{Tk}\ar@/^{0.7pc}/[dd]^{Th}\ar[r]^{\delta_{A}} & T\overline{A}\ar[dd]^{T\overline{h}}\ar@{}[ldd]|-{\quad\cong} &  & TA\ar[dd]_{Tk}\ar[r]^{\delta_{A}} & T\overline{A}\ar@/_{0.7pc}/[dd]_{Tm}\ar@/^{0.7pc}/[dd]^{T\overline{h}}\\
\ltwocell[0]{}{T\lambda} &  & = & \dltwocell[0.3]{r}{\xi} & \ltwocell[0]{}{T\overline{\lambda}}\\
TB\ar[r]_{\delta_{B}} & T\overline{B} &  & TB\ar[r]_{\delta_{B}} & T\overline{B}
}
\]
and thus the left diagram below has the generic factorization
\[
\xymatrix@R=1em{TA\ar@/_{0.7pc}/[dd]_{Tk}\ar@/^{0.7pc}/[dd]^{Th}\ar[r]^{\delta_{A}}\ar@{}[rdd]|-{\quad\quad\cong} & T\overline{A}\ar[r]^{Te_{A}}\ar[dd]|-{T\overline{h}}\ar@{}[rdd]|-{T\cong} & TA\ar[dd]^{Th} &  & TA\ar[dd]_{Tk}\ar[r]^{\delta_{A}} & T\overline{A}\ar@/_{0.7pc}/[dd]_{Tm}\ar@/^{0.7pc}/[dd]^{T\overline{h}}\ar[r]^{Te_{A}}\ar@{}[rdd]|-{\quad\quad T\cong} & TA\ar[dd]^{Th}\\
\ltwocell[0]{}{T\lambda} &  &  & = & \dltwocell[0.3]{r}{\xi} & \ltwocell[0]{}{T\overline{\lambda}}\\
TB\ar[r]_{\delta_{B}} & T\overline{B}\ar[r]_{Te_{B}} & TB &  & TB\ar[r]_{\delta_{B}} & T\overline{B}\ar[r]_{Te_{B}} & TB
}
\]
But this is also the generic factorization of the diagram 
\[
\xymatrix@R=1em{TA\ar[r]^{\delta_{A}}\ar[dd]_{Tk}\ar@{}[rdd]|-{\cong} & T\overline{A}\ar[r]^{Te_{A}}\ar[dd]|-{T\overline{k}}\ar@{}[rdd]|-{T\cong\quad\;\;} & TA\ar@/_{0.7pc}/[dd]_{Tk}\ar@/^{0.7pc}/[dd]^{Th}\\
 &  & \ltwocell[0]{}{T\lambda}\\
TB\ar[r]_{\delta_{B}} & T\overline{B}\ar[r]_{Te_{B}} & TB
}
\]
which has already been factored. By uniqueness of representative generic
factorizations we have $\left(m,\xi\right)=\left(\overline{k},\cong\right)$
as required.

Now, we have the pseudofunctor $\mathbf{P}\colon\int_{\textnormal{lax}}^{X\in\mathscr{B}}\mathfrak{M}_{\left(-\right)}\to\mathscr{A}$,
and will sketch why $\mathbf{P}$ is a left lax $\mathsf{F}$-adjoint
to $G$. To do this, we take our universal 1-cell $\eta_{\left(X,\delta\right)}\colon\left(X,\delta\right)\to GF\left(X,\delta\right)$
at an object $\left(X,\delta\colon X\to LA\right)$ to be the pair
$\left(u_{A},\gamma\right)$ as below.
\[
\xymatrix@=1.5em{X\ar[rr]^{\delta}\ar[dd]_{\delta} &  & TA\ar[ddrr]^{T1_{A}}\ar[dd]_{Tu_{A}}\dltwocell[0.5]{ldld}{\gamma}\\
 &  &  & \dltwocell[0.4]{ld}{T\nu}\\
TA\ar[rr]_{\delta_{A}} &  & T\overline{A}\ar[rr]_{Te_{A}} &  & TA
}
\]
Moreover, for a given 1-cell $\left(f,h,\alpha\right)\colon\left(X,\delta\right)\to GC$
as on the left below, we have
\[
\xymatrix@=1.5em{X\ar[dd]_{f}\ar[rr]^{\delta} &  & TA\ar[dd]|-{Th}\dltwocell[0.5]{ldld}{\alpha}\ar@/^{3.5pc}/@{..>}[dd]|-{T\left(\overline{e_{c}h}u_{A}\right)} &  &  &  &  &  & X\ar[rr]^{\delta}\ar[d]|-{\delta}\ar@/_{3.5pc}/[dd]_{f} &  & TA\ar[d]^{Tu_{A}}\dltwocell[0.5]{ldl}{\gamma}\\
 &  & \ltwocell[0.5]{r}{T\xi} &  &  & = &  &  & TA\ar[rr]_{\delta_{A}}\ar[d]|-{T\left(e_{C}h\right)}\ltwocell[0.8]{l}{Te_{C}\alpha} &  & T\overline{A}\ar[d]^{T\left(\overline{e_{C}h}\right)}\ar@{}[lld]|-{\cong}\\
TC\ar[rr]_{\delta_{c}} &  & T\overline{C} &  &  &  &  &  & TC\ar[rr]_{\delta_{C}} &  & T\overline{C}
}
\]
where $\xi$ is the unique map induced from the fact that the RHS
whiskered by $Te_{C}$ is $Te_{C}\cdot\alpha$. This defines the universal
2-cell
\[
\xymatrix{\left(X,\delta\right)\myar{\left(f,h,\alpha\right)}{rr}\ar[rd]_{\eta_{\left(X,\delta\right)}} & \; & GC\\
 & GA\ar[ur]_{Ge_{c}h}\utwocell[0.7]{u}{Te_{c}\cdot\alpha}
}
\]
where we have a bijection $\beta\mapsto\widetilde{\beta}$ as below
\[
\xymatrix@R=0.6em{\left(X,\delta\right)\myar{\left(f,h,\alpha\right)}{rr}\ar[rdd]_{\eta_{\left(X,\delta\right)}} & \; & GC &  & \left(X,\delta\right)\myar{\left(f,h,\alpha\right)}{rr}\ar[rdd]_{\eta_{\left(X,\delta\right)}} & \; & GC\\
 &  &  & = & \; &  & \utwocell[0.7]{llu}{Te_{C}\cdot\alpha}\\
 & GA\ar[uru]_{G\ell}\utwocell[0.6]{uu}{\beta} &  &  &  & GA\ar[uur]^{Ge_{C}h}\ar@/_{1.3pc}/[uur]_{G\ell} & \ultwocell[0.4]{uul}{G\widetilde{\beta}}
}
\]
or equivalently, a bijection
\[
\xymatrix@=1.5em{X\myar{f}{rr}\ar[dd]_{\delta} &  & TC\ar[dd]^{T\textnormal{id}} &  & X\myar{f}{rr}\ar[dd]_{\delta} &  & TC\ar[dd]^{T\textnormal{id}}\\
 &  &  & = &  & \utwocell[0.5]{ul}{Te_{C}\cdot\alpha}\\
TA\myard{T\ell}{rr}\utwocell[0.5]{urur}{\beta} &  & TC &  & TA\myard{T\ell}{rr}\ar[urur]|-{Te_{c}h} &  & TC\utwocell[0.6]{ul}{T\widetilde{\beta}}
}
\]
as genericity of $\left(h,\alpha\right)$ is respected by composition
with $Te_{C}$. The verification that this bijection satisfies the
required axioms (with the tight maps being defined as above) is left
for the reader.
\end{proof}
Finally, the following provides what is perhaps a more natural definition
of parametric right adjoint (local right adjoint) pseudofunctors,
obtained by applying the above theorem in the setting where $\mathscr{A}$
has a terminal object. In more detail, this is obtained by noting
the reduced form of the spectrum in the presence of a terminal object,
namely $\mathbf{Spec}_{T}\left(X\right)=\mathscr{B}\left(X,T\mathbf{1}\right),$
and applying the spectrum factorization.
\begin{cor}
[Parametric right adjoints] Suppose $\mathscr{A}$ is a bicategory
with a terminal object. Then a pseudofunctor $T\colon\mathscr{A}\to\mathscr{B}$
is lax familial if and only if the canonical projection on the oplax
slice,\footnote{In our convention a morphism $f\nrightarrow g$ in the oplax slice
$\mathscr{B}\sslash T1$ is a morphism $m\colon\textnormal{dom}f\rightarrow\textnormal{dom}g$
and a 2-cell $\alpha\colon f\Rightarrow gm$ in $\mathscr{B}$. The
tight morphisms are those for which $\alpha$ is invertible.} where the underlying pseudo slice defines the tight maps,\footnote{In the case of $\mathscr{A}\sslash1$, the underlying pseudo-slice
coincides with the lax slice. Thus, just as in Theorem \ref{specfacthm},
every morphism in the domain category is tight.}
\[
T_{1}\colon\mathscr{A}\sslash1\to\mathscr{B}\sslash T1
\]
has a left lax $\mathsf{F}$-adjoint. 
\end{cor}

\begin{rem}
There are of course four variants of the above, concerning the case
when $T/T^{\textnormal{op}}/T^{\textnormal{co}}/T^{\textnormal{coop}}$
is familial.
\end{rem}

\section*{Acknowledgements}{

The author would like to thank Richard Garner, Joachim Kock, and the anonymous referee for their feedback on earlier versions of this paper.
}

\bibliographystyle{siam}
\bibliography{references}

\begin{thebibliography}{10}

\bibitem{Bat1998}
{\sc M.~A. Batanin}, {\em Monoidal globular categories as a natural environment
  for the theory of weak {$n$}-categories}, Adv. Math., 136 (1998),
  pp.~39--103.

\bibitem{ben1967}
{\sc J.~B{\'e}nabou}, {\em Introduction to bicategories}, in Reports of the
  {M}idwest {C}ategory {S}eminar, Springer, Berlin, 1967, pp.~1--77.

\bibitem{bettipower}
{\sc R.~Betti and A.~J. Power}, {\em On local adjointness of distributive
  bicategories}, Boll. Un. Mat. Ital. B (7), 2 (1988), pp.~931--947.

\bibitem{Buckley}
{\sc M.~Buckley}, {\em Fibred 2-categories and bicategories}, J. Pure Appl.
  Algebra, 218 (2014), pp.~1034--1074.

\bibitem{CarbJohnFam}
{\sc A.~Carboni and P.~Johnstone}, {\em Connected limits, familial
  representability and {A}rtin glueing}, vol.~5, 1995, pp.~441--459.
\newblock Fifth Biennial Meeting on Category Theory and Computer Science
  (Amsterdam, 1993).

\bibitem{unispans}
{\sc R.~J.~M. Dawson, R.~Par{\'e}, and D.~A. Pronk}, {\em Universal properties
  of {S}pan}, Theory Appl. Categ., 13 (2004), pp.~61--85.

\bibitem{Diers}
{\sc Y.~Diers}, {\em Cat\'egories localisables}, PhD thesis, Universit\'e de
  Paris VI, 1977.

\bibitem{DiersDiag}
\leavevmode\vrule height 2pt depth -1.6pt width 23pt, {\em Spectres et
  localisations relatifs \`a un foncteur}, C. R. Acad. Sci. Paris S\'er. A-B,
  287 (1978), pp.~A985--A988.

\bibitem{gambinokock}
{\sc N.~Gambino and J.~Kock}, {\em Polynomial functors and polynomial monads},
  Math. Proc. Cambridge Philos. Soc., 154 (2013), pp.~153--192.

\bibitem{Gir86}
{\sc J.-Y. Girard}, {\em The system {$F$} of variable types, fifteen years
  later}, Theoret. Comput. Sci., 45 (1986), pp.~159--192.

\bibitem{biequivtri}
{\sc N.~Gurski}, {\em Biequivalences in tricategories}, Theory Appl. Categ., 26
  (2012), pp.~349--384.

\bibitem{kock1995}
{\sc A.~Kock}, {\em Monads for which structures are adjoint to units}, J. Pure
  Appl. Algebra, 104 (1995), pp.~41--59.

\bibitem{BicatTriequiv}
{\sc S.~Lack}, {\em Bicat is not triequivalent to {G}ray}, Theory Appl. Categ.,
  18 (2007), pp.~1--3.

\bibitem{companion}
\leavevmode\vrule height 2pt depth -1.6pt width 23pt, {\em A 2-categories
  companion}, in Towards higher categories, vol.~152 of IMA Vol. Math. Appl.,
  Springer, New York, 2010, pp.~105--191.

\bibitem{Fcaty}
{\sc S.~Lack and M.~Shulman}, {\em Enhanced 2-categories and limits for lax
  morphisms}, Adv. Math., 229 (2012), pp.~294--356.

\bibitem{Lam88}
{\sc F.~Lamarche}, {\em Modelling polymorphism with categories}, PhD thesis,
  McGill University, 1989.

\bibitem{2discfib}
{\sc M.~Lambert}, {\em Discrete 2-fibrations}, arXiv eprint,  (2020).
\newblock Available at \url{https://arxiv.org/abs/2001.11477}.

\bibitem{psmonadicity}
{\sc I.~J. Le~Creurer, F.~Marmolejo, and E.~M. Vitale}, {\em Beck's theorem for
  pseudo-monads}, J. Pure Appl. Algebra, 173 (2002), pp.~293--313.

\bibitem{Foscocoend}
{\sc F.~Loregian}, {\em Coend calculus}, arXiv eprint,  (2019).
\newblock Available at \url{https://arxiv.org/abs/1501.02503}.

\bibitem{Street2Limits}
{\sc R.~Street}, {\em Limits indexed by category-valued {$2$}-functors}, J.
  Pure Appl. Algebra, 8 (1976), pp.~149--181.

\bibitem{fibbicaty}
\leavevmode\vrule height 2pt depth -1.6pt width 23pt, {\em Fibrations in
  bicategories}, Cahiers Topologie G\'{e}om. Diff\'{e}rentielle, 21 (1980),
  pp.~111--160.

\bibitem{StreetPetit}
\leavevmode\vrule height 2pt depth -1.6pt width 23pt, {\em The petit topos of
  globular sets}, J. Pure Appl. Algebra, 154 (2000), pp.~299--315.
\newblock Category theory and its applications (Montreal, QC, 1997).

\bibitem{yonedastructures}
{\sc R.~Street and R.~Walters}, {\em Yoneda structures on 2-categories}, J.
  Algebra, 50 (1978), pp.~350--379.

\bibitem{fibredcategories}
{\sc T.~Streicher}, {\em Fibered categories (a la {J}ean {B}{\'e}nabou)}, arXiv
  eprint,  (2020).
\newblock Available at \url{https://arxiv.org/abs/1801.02927}.

\bibitem{TaylorCandidate}
{\sc P.~Taylor}, {\em The trace factorisation of stable functors}, Preprint,
  (1988).
\newblock Available at \url{https://www.paultaylor.eu/stable/trafsf.pdf}.

\bibitem{WalkerGeneric}
{\sc C.~Walker}, {\em {Generic bicategories}}, arXiv eprint,  (2018).
\newblock Under review; available at \url{http://arxiv.org/abs/1805.01703}.

\bibitem{WalkerYonedaKZ}
\leavevmode\vrule height 2pt depth -1.6pt width 23pt, {\em Yoneda structures
  and {KZ} doctrines}, J. Pure Appl. Algebra, 222 (2018), pp.~1375--1387.

\bibitem{WeberGeneric}
{\sc M.~Weber}, {\em Generic morphisms, parametric representations and weakly
  {C}artesian monads}, Theory Appl. Categ., 13 (2004), pp.~191--234.

\bibitem{Weber2005}
\leavevmode\vrule height 2pt depth -1.6pt width 23pt, {\em Operads within
  monoidal pseudo algebras}, Appl. Categ. Structures, 13 (2005), pp.~389--420.

\bibitem{Webfam}
\leavevmode\vrule height 2pt depth -1.6pt width 23pt, {\em Familial 2-functors
  and parametric right adjoints}, Theory Appl. Categ., 18 (2007), pp.~665--732.

\bibitem{weber}
\leavevmode\vrule height 2pt depth -1.6pt width 23pt, {\em Polynomials in
  categories with pullbacks}, Theory Appl. Categ., 30 (2015), pp.~533--598.

\end{thebibliography}

\end{document}